\documentclass[10pt, reqno]{amsart}

\def\BibTeX{{\rm B\kern-.05em{\sc i\kern-.025em b}\kern-.08em
    T\kern-.1667em\lower.7ex\hbox{E}\kern-.125emX}}

\newtheorem{thm}{Theorem}[section]
\newtheorem{cor}[thm]{Corollary}
\newtheorem{lem}[thm]{Lemma}
\newtheorem{prop}[thm]{Proposition}
\newtheorem{cond}[thm]{Condition}

\theoremstyle{definition}

\theoremstyle{remark}
\newtheorem{rem}{Remark}[section]

\numberwithin{equation}{section}

    \newcommand{\floor}[1]{\lfloor#1\rfloor}

    \newcommand{\EE}{\mathbb{E}}

    \newcommand{\Exp}{\operatorname{E}}
    \newcommand{\E}{\Exp}

    \renewcommand{\Pr}{\operatorname{P}}

    \newcommand{\eqd}{\stackrel{d}{=}}
    \newcommand{\dto}{\xrightarrow{d}}
    \newcommand{\wto}{\xrightarrow{w}}
    \newcommand{\vto}{\xrightarrow{v}}
    \newcommand{\fidi}{\xrightarrow{\text{fidi}}}

    \newcommand{\eind}{\stackrel{d}{=}}
    \newcommand{\rmd}{\mathrm{d}}

\newcommand{\be}{\begin{equation}}
    \newcommand{\ee}{\end{equation}}

\begin{document}

\title[Functional limit theorem for self-normalized linear processes] 
{A functional limit theorem for self-normalized linear processes with random coefficients and i.i.d.~heavy-tailed innovations}

%
\author{Danijel Krizmani\'{c}}

\address{Danijel Krizmani\'{c}\\ Faculty of Mathematics\\
        University of Rijeka\\
        Radmile Matej\v{c}i\'{c} 2, 51000 Rijeka\\
        Croatia}
\email{dkrizmanic@math.uniri.hr}





\begin{abstract}
In this article we derive a self-normalized functional limit theorem for strictly stationary linear processes with i.i.d.~heavy-tailed innovations and random coefficients under the condition that all partial sums of the series of coefficients are a.s.~bounded between zero and the sum of the series. The convergence takes part in the space of c\`{a}dl\`{a}g functions on $[0,1]$ with the Skorokhod $M_{2}$ topology.
\end{abstract}

\maketitle

\section{Introduction}
\label{intro}

We consider linear processes with random coefficients
\begin{equation}\label{e:MArandom}
X_{i} = \sum_{j=0}^{\infty}C_{j}Z_{i-j}, \qquad i \in \mathbb{Z},
\end{equation}
where $(Z_{i})_{i \in \mathbb{Z}}$ is a strictly stationary sequence of regularly varying random variables with index of regular variation $\alpha \in (0,2)$, and
$(C_{i})_{i \geq 0 }$ is a sequence of random variables independent of $(Z_{i})$ such that the above series is a.s.~convergent.
The regular variation property means that
\begin{equation}\label{e:regvar}
 \Pr(|Z_{i}| > x) = x^{-\alpha} L(x), \qquad x>0,
\end{equation}
where $L$ is a slowly varying function at $\infty$. Let $(a_{n})$ be a sequence of positive real numbers tending to infinity such that
\be\label{eq:niz}
n \Pr (|Z_{1}|>a_{n}) \to 1,
\ee
as $n \to \infty$. Then
\begin{equation}
  \label{eq:onedimregvar}
  n \Pr( a_n^{-1} Z_i \in \cdot \, ) \vto \mu( \, \cdot \,) \qquad \textrm{as} \ n \to \infty,
\end{equation}
where the arrow "$\vto$" denotes the vague convergence of measures and $\mu$ is a measure on $\EE = \overline{\mathbb{R}} \setminus \{0\}$ given by
\begin{equation}
\label{eq:mu}
  \mu(\rmd x) = \bigl( p \, 1_{(0, \infty)}(x) + r \, 1_{(-\infty, 0)}(x) \bigr) \, \alpha |x|^{-\alpha-1} \, \rmd x,
\end{equation}
with
\be\label{eq:pq}
p =   \lim_{x \to \infty} \frac{\Pr(Z_i > x)}{\Pr(|Z_i| > x)} \qquad \textrm{and} \qquad
  r =   \lim_{x \to \infty} \frac{\Pr(Z_i < -x)}{\Pr(|Z_i| > x)}.
\ee
One well-known sufficient condition for the a.s.~convergence of the series in (\ref{e:MArandom}) is
\begin{equation}\label{e:momcond}
\sum_{j=0}^{\infty} \mathrm{E} |C_{j}|^{\delta} < \infty \qquad \textrm{for some}  \ \delta < \alpha,\,0 < \delta \leq 1.
\end{equation}
This condition, together with stationarity of the sequence $(Z_{i})$ and $\mathrm{E}|Z_{1}|^{\beta} < \infty$ for every $\beta \in (0,\alpha)$ (which follows from the regular variation property and Karamata's theorem), yields
$$ \mathrm{E}|X_{i}|^{\delta} \leq \sum_{j=0}^{\infty} \mathrm{E}|C_{j}|^{\delta} \mathrm{E}|Z_{i-j}|^{\delta} = \mathrm{E}|Z_{1}|^{\delta} \sum_{j=0}^{\infty}\mathrm{E}|C_{j}|^{\delta} < \infty,$$
which implies the a.s.~convergence of the series in (\ref{e:MArandom}). Another condition that assures the a.s.~convergence of the series in the definition of linear processes with additionally
$$\begin{array}{rl}
 \nonumber \mathrm{E}(Z_{1})=0, & \quad \textrm{if} \ \alpha \in (1,2),\\[0.2em]
 \nonumber Z_{1} \ \textrm{is symmetric}, & \quad \textrm{if} \ \alpha =1,
\end{array}$$
 and a.s.~bounded coefficients
 can be deduced from the results in Astrauskas~\cite{At83}:
$$ \sum_{j =0}^{\infty}c_{j}^{\alpha} L(c_{j}^{-1}) < \infty,$$
where $(c_{j})$ is a sequence of positive real numbers such that $|C_{j}| \leq c_{j}$ a.s.~for all $j$ (c.f.~Balan et al.~\cite{BJL16}).

In the case when the $Z_{i}$'s are independent, under standard regularity conditions and the assumption that all partial sums of the series $C= \sum_{i=0}^{\infty}C_{i}$ are a.s.~bounded between zero and the sum of the series, i.e.
\be\label{eq:InfiniteMAcond}
0 \le \sum_{i=0}^{s}C_{i} \Bigg/ \sum_{i=0}^{\infty}C_{i} \le 1 \ \ \textrm{a.s.} \qquad \textrm{for every} \ s=0, 1, 2 \ldots,
\ee
a functional limit theorem for the corresponding partial sum stochastic process
\be\label{eq:defVn}
V_{n}(t) = \frac{1}{a_{n}}  \sum_{i=1}^{\floor {nt}}X_{i}, \qquad t \in [0,1],
\ee
holds in the space $D([0,1], \mathbb{R})$ of real-valued c\`{a}dl\`{a}g functions on $[0,1]$ with the Skorokhod $M_{2}$ topology, see Krizmani\'{c}~\cite{Kr19}. This result was then extended in Krizmani\'{c}~\cite{Kr22-1} to the case when the innovations $Z_{i}$ are weakly dependent, in the sense that $(Z_{i})$ is a strongly mixing sequence which satisfies the local dependence condition $D'$ as is given in Davis~\cite{Da83}:
 \begin{equation}\label{e:D'cond}
 \lim_{k \to \infty} \limsup_{n \to \infty}~n \sum_{i=1}^{\lfloor n/k \rfloor} \Pr \bigg( \frac{|Z_{0}|}{a_{n}} > x, \frac{|Z_{i}|}{a_{n}} >x \bigg) = 0 \qquad \textrm{for all} \ x >0.
 \end{equation}
 More precisely,
 \be\label{eq:fconvVn}
 V_{n}(\,\cdot\,) \dto \widetilde{C} V(\,\cdot\,) \qquad \textrm{as} \ n \to \infty,
\ee
 in $D([0,1], \mathbb{R})$ endowed with the $M_{2}$ topology,
where $V$ is an $\alpha$--stable L\'{e}vy process with characteristic triple $(0, \mu, b)$, with $\mu$ as in $(\ref{eq:mu})$,
$$ b = \left\{ \begin{array}{cc}
                                   0, & \quad \alpha = 1,\\[0.4em]
                                   (p-r)\frac{\alpha}{1-\alpha}, & \quad \alpha \in (0,1) \cup (1,2),
                                 \end{array}\right.$$
 and $\widetilde{C}$ is a random variable, independent of $V$, such that $\widetilde{C} \eind C$.
Condition (\ref{e:D'cond}) is satisfied for i.i.d.~regularly varying sequences, and together with the strong mixing property it assures that, as in the i.i.d.~case, the extremes of the sequence $(Z_{i})$ are isolated.
 This corresponds to the situation when the extremal index $\theta$ of the sequence $(Z_{i})$, which can be interpreted as the reciprocal mean cluster size of large exceedances (c.f.~Hsing et al.~\cite{HHL88}), is equal to $1$. Recall here that a strictly stationary sequence of random variables $(\xi_{n})$ has extremal index $\theta$ if for every $\tau >0$ there exists a sequence of real numbers $(u_{n})$ such that
\begin{equation*}\label{eq:eindex}
 \lim_{n \to \infty} n \Pr ( \xi_{1} > u_{n}) \to \tau \qquad \textrm{and} \qquad \lim_{n \to \infty} \Pr \bigg( \max_{1 \leq i \leq n} \xi_{i} \leq u_{n} \bigg) \to e^{-\theta \tau}.
\end{equation*}
It holds that $\theta \in [0,1]$. Recall also that a sequence $(\xi_{n})$ is strongly mixing if $\alpha (n) \to 0$ as $n \to \infty$, where
$$\alpha (n) = \sup \{|\Pr (A \cap B) - \Pr(A) \Pr(B)| : A \in \mathcal{F}_{-\infty}^{0}, B \in \mathcal{F}_{n}^{\infty} \}$$
and $\mathcal{F}_{k}^{l} = \sigma( \{ \xi_{i} : k \leq i \leq l \} )$ for $-\infty \leq k \leq l \leq \infty$. For some related results on limit theory for moving averages with random coefficients we refer to Hult and Samorodnitsky~\cite{HuSa08} and Kulik~\cite{Ku06}.

In this paper we study self-normalized partial sum processes
$\{ \zeta_{n}^{-1}\sum_{i=1}^{\floor {nt}}X_{i},\,t \in [0,1] \}$
as random elements of the space $D([0,1], \mathbb{R})$,
where $\zeta_{n}^{2}=X_{1}^{2} + \ldots + X_{n}^{2}$. Self-normalized processes arise naturally in the study of stochastic integrals, martingale inequalities and various statistical studies. Self-normalization is often an appropriate technique in deriving central limit theorems for partial sums $X_{1}+ \ldots + X_{n}$, $n \in \mathbb{N}$, and functional limit theorems for processes build from partial sums. Functional limit theorems (or invariance principles) play an important role in statistics and econometrics, for example in deriving asymptotic distributions of unit-root test statistics, see Wu~\cite{Wu06}. They have been studied by many authors in the literature. Among them we mention Cs\"{o}rg\H{o} et al.~\cite{CSW03}, Peligrad and Sang~\cite{PeSa12}, Ra\v{c}kauskas and Suquet~\cite{RaSu01} for the case of independent random variables, Balan and Kulik~\cite{BaKu09}, Choi and Moon~\cite{ChMo10} for $\phi$--mixing sequences, Choi et al.~\cite{CSM17}, Kulik~\cite{Ku06-2} and Ra\v{c}kauskas and Suquet~\cite{RaSu11} for a special class of dependent random variables, i.e.~linear processes generated by independent or strictly stationary $\rho$--mixing innovations with infinite variance. For a survey on self-normalized limit theorems we refer to Shao and Wang~\cite{ShWa13}, and for some results and applications of self-normalized process for dependent random variables we refer to a survey paper by de la Pe\~{n}a et al.~\cite{PKL07}.

 Our aim in this paper is to derive a functional limit theorem for the self-normalized partial sum process of linear processes $(X_{i})_{i \in \mathbb{Z}}$ with i.i.d.~heavy-tailed innovations and random coefficients. More precisely, under condition (\ref{eq:InfiniteMAcond}) and appropriate moment conditions on the sequence of coefficients we will show that
\begin{equation}\label{e:SNuvod}
  \frac{1}{\zeta_{n}}  \sum_{i=1}^{\floor {n\,\cdot}}X_{i} \dto \frac{C^{(1)}V(\,\cdot\,)}{\sqrt{C^{(2)}W(1)}} \qquad \textrm{as} \ n \to \infty,
\end{equation}
in $D([0,1], \mathbb{R})$ equipped with the Skorokhod $M_{2}$ topology, where $V$ and $W$ are stable L\'{e}vy process, and $C^{(1)}$ and $C^{(2)}$ are certain random variables related to the sequence of coefficients of the underlying linear process to be specified in subsequent sections. Some auxiliary results in Sections~\ref{S:pomocno} and~\ref{S:jointconv} will be derived under slightly more general conditions than independence for the innovations, namely under condition $D'$ and strong mixing. The independence property for the innovations, among other assumptions, will be crucial in deriving functional limit theorems in Section~\ref{S:FiniteMA} and the subsequent sections.
In order to prove relation (\ref{e:SNuvod}) we will first establish functional convergence of the joint stochastic process
$$ L_{n}(t) = \bigg( \frac{1}{a_{n}}  \sum_{i=1}^{\floor {nt}}X_{i}, \frac{1}{a_{n}^{2}}  \sum_{i=1}^{\floor {nt}}X_{i}^{2} \bigg), \qquad t \in [0,1]$$
in the space of $\mathbb{R}^{2}$--valued c\`{a}dl\`{a}g functions on
$[0,1]$ with the weak Skorokhod $M_{2}$ topology, and then we will apply the continuous mapping theorem to obtain (\ref{e:SNuvod}).
In establishing this joint convergence we will turn our attention first to finite order linear processes, and then to the general case of infinite order linear processes. For $Z_{1}$ we assume the already mentioned standard regularity conditions:
  \begin{eqnarray}\label{e:oceknula}
    \mathrm{E} Z_{1}=0, & & \textrm{if} \ \ \alpha \in (1,2),  \\
    Z_{1} \ \textrm{is symmetric}, & & \textrm{if} \ \ \alpha=1.\label{e:sim}
  \end{eqnarray}
For infinite order moving averages, beside condition (\ref{e:momcond}) we will require also some other moment conditions, which will be specified latter in Section~\ref{S:InfiniteMA}.

The paper is organized as follows. In Section~\ref{S:Mtop} we recall the definitions of Skorokhods's $M_{1}$ and $M_{2}$ topologies and list several lemmas related to $M_{1}$ and $M_{2}$ continuity of addition, multiplication and division that we are going to use in the following sections. Section~\ref{S:pomocno} is devoted to regular variation and point processes convergence. The latter is then transferred in Section~\ref{S:jointconv} to the joint functional convergence of an auxiliary partial sum stochastic process in the space of $\mathbb{R}^{2}$--valued c\`{a}dl\`{a}g functions on
$[0,1]$ with the weak Skorokhod $M_{1}$ topology. We use this result in Section~\ref{S:FiniteMA} to establish functional convergence of the joint stochastic process $L_{n}(\,\cdot\,)$ with respect to the weak $M_{2}$ topology for finite order linear processes with random coefficients and i.i.d.~regularly varying innovations. Then in Section~\ref{S:InfiniteMA} we extend this to infinite order linear processes. Finally, in Section~\ref{S:Selfnormalized} we derive functional convergence of the self-normalized partial sum process of linear processes with random coefficients and i.i.d.~heavy-tailed innovations.

\section{Skorokhod's $M_{1}$ and $M_{2}$ topologies}\label{S:Mtop}

We start with the definition of the Skorokhod weak $M_{2}$ topology in a general space $D([0,1], \mathbb{R}^{d})$ of $\mathbb{R}^{d}$--valued c\`{a}dl\`{a}g functions on
$[0,1]$. It is standardly defined using completed graphs and their parametric representations.

For $x \in D([0,1],
\mathbb{R}^{d})$ the completed (thick) graph of $x$ is the set
\[
  G_{x}
  = \{ (t,z) \in [0,1] \times \mathbb{R}^{d} : z \in [[x(t-), x(t)]]\},
\]
where $x(t-)$ is the left limit of $x$ at $t$ and $[[a,b]]$ is the product segment, i.e.
$[[a,b]]=[a_{1},b_{1}] \times [a_{2},b_{2}] \ldots \times [a_{d},b_{d}]$
for $a=(a_{1}, a_{2}, \ldots, a_{d}), b=(b_{1}, b_{2}, \ldots, b_{d}) \in
\mathbb{R}^{d}$. We define an
order on the graph $G_{x}$ by saying that $(t_{1},z_{1}) \le
(t_{2},z_{2})$ if either (i) $t_{1} < t_{2}$ or (ii) $t_{1} = t_{2}$
and $|x_{j}(t_{1}-) - z_{1j}| \le |x_{j}(t_{2}-) - z_{2j}|$
for all $j=1, 2, \ldots, d$. The relation $\le$ induces only a partial
order on the graph $G_{x}$. A weak $M_{2}$ parametric representation
of the graph $G_{x}$ is a continuous function $(r,u)$
mapping $[0,1]$ into $G_{x}$, such that $r$ is nondecreasing with $r(0)=0$, $r(1)=1$ and $u(1)=x(1)$ ($r$ is the
time component and $u$ the spatial component). Denote by $\Pi_{w}^{M_{2}}(x)$ the set of weak $M_{2}$
parametric representations of the graph $G_{x}$. For $x_{1},x_{2}
\in D([0,1], \mathbb{R}^{d})$ define
\[
  d_{w}(x_{1},x_{2})
  = \inf \{ \|r_{1}-r_{2}\|_{[0,1]} \vee \|u_{1}-u_{2}\|_{[0,1]} : (r_{i},u_{i}) \in \Pi_{w}^{M_{2}}(x_{i}), i=1,2 \},
\]
where $\|x\|_{[0,1]} = \sup \{ \|x(t)\| : t \in [0,1] \}$ for $x \colon [0,1] \to \mathbb{R}^{k}$, with $\| \cdot \|$ denoting the max-norm on $\mathbb{R}^{k}$. Now we define the weak $M_{2}$ topology sequentially by saying that a sequence $(x_{n})_{n}$ converges to $x$ in $D([0,1], \mathbb{R}^{d})$ in the weak Skorokhod $M_{2}$
topology if $d_{w}(x_{n},x)\to 0$ as $n \to \infty$.

If we replace the graph $G_{x}$ with the completed (thin) graph
\[
  \Gamma_{x}
  = \{ (t,z) \in [0,1] \times \mathbb{R}^{d} : z= \lambda x(t-) + (1-\lambda)x(t) \ \text{for some}\ \lambda \in [0,1] \},
\]
and a weak $M_{2}$ parametric representation with a strong $M_{2}$ parametric representation (i.e. a continuous function $(r,u)$ mapping $[0,1]$ onto $\Gamma_{x}$ such that $r$ is nondecreasing), then we obtain the standard (or strong) $M_{2}$ topology. This topology is stronger than the weak $M_{2}$ topology, but they coincide if $d=1$.

Note that in $M_{2}$ parametric representations $(r,u)$ we require that only the time component $r$ is nondecreasing. If we also require that the spatial component $u$ is nondecreasing, then we obtain $M_{1}$ parametric representations and (weak and strong) Skorokhod $M_{1}$ topologies, which are stronger than the corresponding $M_{2}$ topologies.

Often the following characterization of the $M_{2}$ topology with the Hausdorff metric on the spaces of graphs is useful. For $x_{1},x_{2} \in D([0,1], \mathbb{R}^{d})$, the $M_{2}$ distance between $x_{1}$ and $x_{2}$ is given by
\begin{equation}\label{e:defM2H}
 d_{M_{2}}(x_{1}, x_{2}) = \bigg(\sup_{a \in \Gamma_{x_{1}}} \inf_{b \in \Gamma_{x_{2}}} d(a,b) \bigg) \vee \bigg(\sup_{a \in \Gamma_{x_{2}}} \inf_{b \in \Gamma_{x_{1}}} d(a,b) \bigg),
\end{equation}
where $d$ is the metric induced by the maximum norm on $\mathbb{R}^{d+1}$.
The metric $d_{M_{2}}$ induces the strong $M_{2}$ topology.
The weak $M_{2}$ topology on $D([0,1], \mathbb{R}^{d})$
coincides with the (product) topology induced by the metric
\begin{equation}\label{e:defdp}
 d_{p}^{M_{2}}(x_{1},x_{2})= \max_{j=1,\ldots,d}d_{M_{2}}(x_{1j},x_{2j})
\end{equation}
 for $x_{i}=(x_{i1}, \ldots, x_{id}) \in D([0,1],
 \mathbb{R}^{d})$, $i=1,2$. For detailed discussion of the strong and weak $M_{2}$ topologies we refer to Whitt~\cite{Whitt02}, sections 12.10--12.11.
 For simplicity of notation let $D^{d} \equiv D([0,1], \mathbb{R}^{d})$.

Similar to relation (\ref{e:defdp}) for the weak $M_{2}$ topology, the weak $M_{1}$ topology on $D^{d}$ coincides with the topology induced by the metric
\begin{equation}\label{e:defdpM1}
 d_{p}^{M_{1}}(x_{1},x_{2})= \max_{j=1,\ldots,d}d_{M_{1}}(x_{1j},x_{2j})
\end{equation}
 for $x_{i}=(x_{i1}, \ldots, x_{id}) \in D^{d}$, $i=1,2$ (see Whitt~\cite{Whitt02}, Theorem 12.5.2). Here $d_{M_{1}}$ denotes the $M_{1}$ metric on $D^{1}$, defined by
$$ d_{M_{1}}(y_{1},y_{2})
  = \inf \{ \|r_{1}-r_{2}\|_{[0,1]} \vee \|u_{1}-u_{2}\|_{[0,1]} : (r_{i},u_{i}) \in \Pi^{M_{1}}(y_{i}), i=1,2 \}$$
for $y_{1},y_{2} \in D^{1}$, where $\Pi^{M_{1}}(y)$ is the set of $M_{1}$ parametric representations of the completed graph $\Gamma_{y}$, i.e.~continuous nondecreasing functions $(r,u)$ mapping $[0,1]$ onto $\Gamma_{y}$.

Denote by $C^{\uparrow}_{0}([0,1], \mathbb{R})$ the subset of functions $x$ in $D^{1}$ that are continuous and nondecreasing with $x(0)>0$. Since a continuous function on a compact set is bounded and attains its minimum value, it holds that
$$ C^{\uparrow}_{0}([0,1], \mathbb{R}) = \bigcup_{k=1}^{\infty}C^{\uparrow}_{k}([0,1], \mathbb{R}),$$
where $C^{\uparrow}_{k}([0,1], \mathbb{R})$ is the subset of functions $x$ in $C^{\uparrow}_{0}([0,1], \mathbb{R})$ with $x(0) \geq 1/k$.
Since $C^{\uparrow}_{k}([0,1], \mathbb{R})$ is a closed set with respect to the Skorokhod $J_{1}$ topology (cf.~Kulik and Soulier~\cite{KuSo}, Proposition C.2.4), it is also a measurable subset of $D^{1}$ with the $J_{1}$ topology. The Borel $\sigma$--fields associated with the non-uniform Skorokhod topologies all coincide with the Kolmogorov $\sigma$--field generated by the projection maps (see Theorem 11.5.2 and Theorem 11.5.3 in Whitt~\cite{Whitt02}), and hence the set $C^{\uparrow}_{k}([0,1], \mathbb{R})$ is measurable with the $M_{i}$ topology for each $k$ ($i=1,2$), and therefore the same holds for $C^{\uparrow}_{0}([0,1], \mathbb{R})$.

In the next sections we will use the following three lemmas.
The first one is about preservation of weak $M_{1}$ convergence of stochastic processes under transformations that add certain c\`adl\`ag functions to the components of the underlying processes. This results is a simple consequence of $M_{1}$ continuity of addition, the continuous mapping theorem and Slutsky's theorem. It can be proven similarly as Lemma 1 in Krizmani\'{c}~\cite{Kr18}. The remaining two lemmas deal with continuity of multiplication and division of two c\`{a}dl\`{a}g functions. The first one is based on Theorem 13.3.2 in Whitt~\cite{Whitt02}, and the second one follows from the fact that for monotone functions $M_{1}$ convergence is equivalent to point-wise convergence in a dense subset of $[0,1]$ including $0$ and $1$. Denote by $\textrm{Disc}(x)$ the set of discontinuity points of $x \in D^{1}$.

\begin{lem}\label{l:weakM1transf}
Let $(A_{n}, B_{n})$, $n=0,1,2,\ldots$, be stochastic processes in $D^{2}$ such that, as $n \to \infty$,
\begin{equation}\label{e:lemM2}
(A_{n}, B_{n}) \dto (A_{0}, B_{0})
\end{equation}
in $D^{2}$ with the weak $M_{1}$ topology. Suppose $x_{n}$ and $y_{n}$, $n=0,1,2,\ldots$, are elements of $D^{1}$ with $x_{0}$ and $y_{0}$ being continuous, such that, as $n \to \infty$,
$$x_{n}(t) \to x_{0}(t) \qquad \textrm{and} \qquad y_{n}(t) \to y_{0}(t) $$
uniformly in $t$. Then
$$ (A_{n} +x_{n}, B_{n} + y_{n}) \dto (A_{0} + x_{0}, B_{0} + y_{0})$$
in $D^{2}$ with the weak $M_{1}$ topology.
\end{lem}

\begin{lem}\label{l:contmultpl}
Suppose that $x_{n} \to x$ and $y_{n} \to y$ in $D^{1}$ with the $M_{1}$ or $M_{2}$ topology. If for each $t \in \textrm{Disc}(x) \cap \textrm{Disc}(y)$, $x(t)$, $x(t-)$, $y(t)$ and $y(t-)$ are all nonnegative and $[x(t)-x(t-)][y(t)-y(t-)] \geq 0$, then $x_{n}y_{n} \to xy$ in $D^{1}$ with the same topology, where $(xy)(t) = x(t)y(t)$ for $t \in [0,1]$.
\end{lem}

\begin{lem}\label{l:M1div}
The function $h \colon D^{1} \times C^{\uparrow}_{0}([0,1], \mathbb{R}) \to D^{1}$ defined by
$h(x,y)= x/y$, where
$$ \Big(\frac{x}{y} \Big)(t) = \frac{x(t)}{y(t)}, \qquad t \in [0,1],$$
is continuous
when $D^{1} \times C^{\uparrow}_{0}([0,1], \mathbb{R})$ is endowed with the weak $M_{i}$ topology and $D^{1}$ is endowed with the standard $M_{i}$ topology for $i=1,2$.
\end{lem}
\begin{proof}
Take an arbitrary $(x,y) \in  D^{1} \times C^{\uparrow}_{0}([0,1], \mathbb{R})$, and assume first that $(x_{n}, y_{n}) \to (x,y)$ in $D^{1} \times C^{\uparrow}_{0}([0,1], \mathbb{R})$ with respect to the weak $M_{1}$ topology, i.e.
\begin{equation*}\label{e:M1dense}
d_{p}^{M_{1}}((x_{n},y_{n}), (x,y)) =
\max \{ d_{M_{1}}(x_{n}, x), d_{M_{1}}(y_{n}, y)) \to 0 \qquad \textrm{as} \ n \to \infty.
\end{equation*}
In order for $f$ to be continuous we have to show that
$h(x_{n}, y_{n}) \to h(x,y)$ in $D^{1}$ with respect to the $M_1$ topology. Using the fact that for monotone functions $M_{1}$ convergence is equivalent to point-wise convergence in a dense subset of $[0,1]$ including $0$ and $1$ (see Whitt~\cite{Whitt02}, Corollary 12.5.1), from
$d_{M_{1}}(y_{n}, y) \to 0$ we obtain $y_{n}(t) \to y(t)$ for all $t$ in a dense subset of $[0,1]$ including $0$ and $1$. Note that for all such $t$ it holds also
$$ \frac{1}{y_{n}(t)} \to \frac{1}{y(t)} \qquad \textrm{as} \ n \to \infty.$$
Since the functions $1/y$ and $1/y_{n}$ for all $n \in \mathbb{N}$ are monotone, a new application of the above mentioned characterization of $M_{1}$ convergence for monotone functions yields $1/y_{n} \to 1/y$ in $D^{1}$ with the $M_{1}$ topology. Note that $1/y$ is a continuous function, and hence we can apply Lemma~\ref{l:contmultpl} to obtain
$ x_{n}/y_{n} \to x/y$, i.e. $h(x_{n}, y_{n}) \to h(x,y)$ as $n \to \infty$, in $D^{1}$ with the $M_{1}$ topology. Therefore the function $h$ is continuous at $(x,y)$ when $D^{1} \times C^{\uparrow}_{0}([0,1], \mathbb{R})$ is endowed with the weak $M_{1}$ topology and $D^{1}$ is endowed with the standard $M_{1}$ topology.

 We continue the proof by showing the same for the $M_{2}$ topology. Let $(x_{n}, y_{n}) \to (x,y)$ in $D^{1} \times C^{\uparrow}_{0}([0,1], \mathbb{R})$ with respect to the weak $M_{2}$ topology, i.e.
\begin{equation}\label{e:M2dense}
d_{p}^{M_{2}}((x_{n},y_{n}), (x,y)) =
\max \{ d_{M_{2}}(x_{n}, x), d_{M_{2}}(y_{n}, y)) \to 0 \qquad \textrm{as} \ n \to \infty.
\end{equation}
By Remark 12.8.1 in Whitt~\cite{Whitt02} the following metric is a complete metric topologically equivalent to $d_{M_{1}}$:
$$ {d_{M_{1}}^{*}}(x_{1}, x_{2}) = d_{M_{2}}(x_{1}, x_{2}) + \lambda (\widehat{\omega}(x_{1},\cdot), \widehat{\omega}(x_{2},\cdot)),$$
where $\lambda$ is the L\'{e}vy metric on a space of distributions
$$ \lambda (F_{1},F_{2}) = \inf \{ \epsilon >0 : F_{2}(x-\epsilon) - \epsilon \leq F_{1}(x) \leq F_{2}(x+\epsilon) + \epsilon \ \ \textrm{for all} \ x \},$$
and
$$ \widehat{\omega}(x,z) = \left\{ \begin{array}{cc}
                                   \omega(x,e^{z}), & \quad z<0,\\[0.4em]
                                   \omega(x,1), & \quad z \geq 0,
                                 \end{array}\right.$$
with
$$ \omega (x,\delta) = \sup_{0 \leq t \leq 1} \ \sup_{0 \vee (t-\delta) \leq t_{1} < t_{2} < t_{3} \leq (t+\delta) \wedge 1} \{\| x(t_{2}) - [x(t_{1}), x(t_{3})] \| \}$$
for $x \in D^{1}$ and $\delta >0$.
Here $\|z-A\|$ denotes the distance between a point $z$ and a subset $A \subseteq \mathbb{R}$.
Since $y$ and $y_{n}$ are monotone functions, for $t_{1} < t_{2} < t_{3}$ it holds that
$ \|y(t_{2}) - [y(t_{1}), y(t_{3})] \|=0$, which yields $\omega(y, \delta)=0$ for all $\delta>0$, and similarly $\omega(y_{n}, \delta)=0$. Hence $\lambda (y_{n}, y)=0$, and $d_{M_{1}}^{*}(y_{n}, y) = d_{M_{2}}(y_{n}, y)$. Now from (\ref{e:M2dense}) we obtain
$ d_{M_{1}}^{*}(y_{n}, y) \to 0$ as $n \to \infty$, that is, $y_{n} \to y$ in $D^{1}$ with the $M_{1}$ topology. By the first part of the proof this implies $1/y_{n} \to 1/y$ in $D^{1}$ with the $M_{1}$ topology, and since the functions $1/y$ and $1/y_{n}$ are monotone, we obtain
$$ d_{M_{2}}(1/y_{n}, 1/y) =d_{M_{1}}^{*}(1/y_{n}, 1/y) \to 0 \qquad \textrm{as} \ n \to \infty.$$
Finally, a new application of Lemma~\ref{l:contmultpl} yields $h(x_{n}, y_{n}) \to h(x,y)$ as $n \to \infty$, in $D^{1}$ with the $M_{2}$ topology, which means that $h$ is also continuous when $D^{1} \times C^{\uparrow}_{0}([0,1], \mathbb{R})$ is endowed with the weak $M_{2}$ topology and $D^{1}$  with the standard $M_{2}$ topology.
\end{proof}

For $a \in \mathbb{R}$ denote by $\widehat{a}$ the constant function in $D^{1}$ defined by $\widehat{a}(t)=a$ for all $t \in [0,1]$. Let $D_{m}([0,1], \mathbb{R})$ be the subset of functions $x$ in $D^{1}$ that are monotone with $x(0) \geq 0$. This set is a measurable subset of $D^{1}$ with the Skorokhod $M_{1}$ and $M_{2}$ topologies (see Lemma 13.2.3 in Whitt~\cite{Whitt02}).

\begin{cor}\label{c:M2cont}
The function $f \colon D^{1} \times D_{m}([0,1], \mathbb{R}) \to D^{1} \times D_{m}([0,1], \mathbb{R}) $ defined by
$$ f(x,y) = (x, \widehat{y(1)}), \qquad (x, y) \in D^{1} \times D_{m}([0,1], \mathbb{R}),$$
is continuous when $D^{1} \times D_{m}([0,1], \mathbb{R})$ is endowed with the weak $M_{2}$ topology.
\end{cor}
\begin{proof}
Take an arbitrary $(x,y) \in  D^{1} \times D_{m}([0,1], \mathbb{R})$, and assume that $(x_{n}, y_{n}) \to (x,y)$ in $D^{1} \times D_{m}([0,1], \mathbb{R})$ with respect to the weak $M_{2}$ topology. This implies $y_{n} \to y$ in $D^{1}$ with the $M_{2}$ topology, and hence as in the proof of Lemma~\ref{l:M1div} we obtain $y_{n} \to y$ in $D^{1}$ with the $M_{1}$ topology. Since $M_{1}$ convergence implies pointwise convergence at the endpoints, in particular at $t=1$ (see Whitt~\cite{Whitt02}, Theorem 12.5.1), $y_{n}(1) \to y(1)$ as $n \to \infty$. Therefore using the fact that the metric $d_{M_{2}}$ is bounded above by the uniform metric (see Theorem 12.10.3 in Whitt~\cite{Whitt02}) we have
$$ d_{M_{2}}(\widehat{y_{n}(1)}, \widehat{y(1)}) \leq \|\widehat{y_{n}(1)} - \widehat{y(1)}\|_{[0,1]} = |y_{n}(1)-y(1)| \to 0 \qquad \textrm{as} \ n \to \infty.$$
This, together with $d_{M_{2}}(x_{n}, x) \to 0$, yields $d_{p}^{M_{2}}((x_{n}, \widehat{y_{n}(1)}), (x, \widehat{y(1)})) \to 0$ as $n \to \infty$, and hence we conclude that $f$ is continuous at $(x,y)$ with respect to the weak $M_{2}$ topology.
\end{proof}

\section{Joint regular variation, point processes and dependence conditions}
\label{S:pomocno}

We say that a strictly stationary sequence of random variables $(\xi_{n})_{n \in \mathbb{Z}}$ is (jointly) regularly varying with index
$\alpha >0$ if for any nonnegative integer $k$ the
$k$-dimensional random vector $\xi = (\xi_{1}, \ldots , \xi_{k})$ is
multivariate regularly varying with index $\alpha$, i.e.\ there
exists a random vector $\Theta$ on the unit sphere
$\mathbb{S}^{k-1} = \{ x \in \mathbb{R}^{k} : \|x\|=1 \}$ such
that for every $u >0$, as $x \to \infty$,
 \begin{equation}\label{e:regvar1}
   \frac{\Pr(\|\xi\| > ux,\,\xi / \| \xi \| \in \cdot \, )}{\Pr(\| \xi \| >x)}
    \wto u^{-\alpha} \Pr( \Theta \in \cdot \,),
 \end{equation}
where the arrow ''$\wto$'' denotes the weak convergence of finite measures.
There is a convenient characterization of joint regular variation due to Basrak and Segers~\cite{BaSe}:~it is necessary and
sufficient that there exists a process $(Y_n)_{n \in \mathbb{Z}}$
with $\Pr(|Y_0| > y) = y^{-\alpha}$ for $y \geq 1$ such that, as $x
\to \infty$,
\begin{equation}\label{e:tailprocess1}
  \bigl( (x^{-1}\ \xi_n)_{n \in \mathbb{Z}} \, \big| \, | \xi_0| > x \bigr)
  \fidi (Y_n)_{n \in \mathbb{Z}},
\end{equation}
where "$\fidi$" denotes convergence of finite-dimensional
distributions. The process $(Y_{n})$ is called
the tail process of $(\xi_{n})$.

Let $(Z_{i})_{i \in \mathbb{Z}}$ be a strictly stationary and strongly mixing sequence of regularly varying random variables with index $\alpha \in (0,2)$, such that the local dependence condition $D'$ and conditions $(\ref{e:oceknula})$ and $(\ref{e:sim})$ hold.
Condition $D'$ and strong mixing imply that the extremes of the sequence $(Z_{i})$ are isolated, i.e.~$\theta=1$ (see Leadbetter and Rootz\'{e}n~\cite{LeRo88}, page 439, and Leadbetter et al.~\cite{LLR83}, Theorem 3.4.1). They also imply that $(Z_{i})$ is jointly regularly varying with the tail process $(Y_{i})$ being the same as in the i.i.d.~case, that is, $Y_{i}=0$ for $i \neq 0$, and $Y_{0}$  as described above (Basrak et al.~\cite{BKS}, Example 4.1). This in particular means that $(Y_{i})$ has no two values of the opposite sign.

Define the time-space point processes
\begin{equation*}\label{E:ppspacetime}
 N_{n} = \sum_{i=1}^{n} \delta_{(i / n,\,Z_{i} / a_{n})} \qquad \textrm{for all} \ n\in \mathbb{N},
\end{equation*}
with $a_{n}$ as in (\ref{eq:niz}). The point process convergence for the sequence $(N_{n})$ on the space $[0,1] \times \EE$ was obtained by Basrak and Tafro~\cite{BaTa16} under joint regular variation and the following two weak dependence conditions.

\begin{cond}\label{c:mixcond1}
There exists a sequence of positive integers $(r_{n})$ such that $r_{n} \to \infty $ and $r_{n} / n \to 0$ as $n \to \infty$ and such that for every nonnegative continuous function $f$ on $[0,1] \times \mathbb{E}$ with compact support, denoting $k_{n} = \lfloor n / r_{n} \rfloor$, as $n \to \infty$,
\begin{equation}\label{e:mixcon}
 \E \biggl[ \exp \biggl\{ - \sum_{i=1}^{n} f \biggl(\frac{i}{n}, \frac{Z_{i}}{a_{n}}
 \biggr) \biggr\} \biggr]
 - \prod_{k=1}^{k_{n}} \E \biggl[ \exp \biggl\{ - \sum_{i=1}^{r_{n}} f \biggl(\frac{kr_{n}}{n}, \frac{Z_{i}}{a_{n}} \biggr) \biggr\} \biggr] \to 0.
\end{equation}
\end{cond}

\begin{cond}\label{c:mixcond2}
There exists a sequence of positive integers $(r_{n})$ such that $r_{n} \to \infty $ and $r_{n} / n \to 0$ as $n \to \infty$ and such that for every $u > 0$,
\begin{equation}
\label{e:anticluster}
  \lim_{m \to \infty} \limsup_{n \to \infty}
  \Pr \biggl( \max_{m \leq |i| \leq r_{n}} | Z_{i} | > ua_{n}\,\bigg|\,| Z_{0}|>ua_{n} \biggr) = 0.
\end{equation}
\end{cond}
Condition~\ref{c:mixcond1} is implied by the strong mixing property (see Krizmani\'{c}~\cite{Kr10},~\cite{Kr16}). Condition~\ref{c:mixcond2} follows from condition $D'$, for the latter implies
\[
    \lim_{n \to \infty} n \sum_{i=1}^{r_{n}} \Pr \bigg( \frac{|Z_{0}|}{a_{n}} > u,
    \frac{|Z_{i}|}{a_{n}} > u \bigg) = 0 \quad \textrm{for all} \ u
    >0,
\]
for any sequence of positive integers $(r_{n})$ such that $r_{n} \to \infty$ and $r_{n} / n \to 0$ as $n \to \infty$.
Therefore, in our case, by Theorem 3.1 in Basrak and Tafro~\cite{BaTa16}, as $n \to \infty$,
\begin{equation}\label{e:BaTa}
N_{n} \dto N = \sum_{i}\sum_{j}\delta_{(T_{i}, P_{i}\eta_{ij})}
\end{equation}
in $[0,1] \times \EE$, where $\sum_{i=1}^{\infty}\delta_{(T_{i}, P_{i})}$ is a Poisson process on $[0,1] \times (0,\infty)$
with intensity measure $Leb \times \nu$ where $\nu(\rmd x) = \alpha
x^{-\alpha-1}1_{(0,\infty)}(x)\,\rmd x$, and $(\sum_{j= 1}^{\infty}\delta_{\eta_{ij}})_{i}$ is an i.i.d.~sequence of point processes in $\EE$ independent of $\sum_{i}\delta_{(T_{i}, P_{i})}$ and with common distribution equal to the distribution of $\sum_{j}\delta_{\widetilde{Y}_{j}/L(\widetilde{Y})}$, where $L(\widetilde{Y})= \sup_{j \in \mathbb{Z}}|\widetilde{Y}_{j}|$ and $\sum_{j}\delta_{\widetilde{Y}_{j}}$ is distributed as $( \sum_{j \in \mathbb{Z}} \delta_{Y_j} \,|\, \sup_{i \le -1} | Y_i| \le 1).$ Taking into account the form of the tail process $(Y_{i})$ it holds that $N=\sum_{i}\delta_{(T_{i}, P_{i}\eta_{i0})}$ with $|\eta_{i0}|=1$. Further, by (\ref{eq:pq}) and (\ref{e:tailprocess1}) we obtain
\begin{eqnarray*}
\Pr(\eta_{i0}=1) &= &\Pr(Y_{0}>0)=\Pr(Y_{0}>1) = \lim_{x \to \infty}  \Pr \bigl(x^{-1}Z_{0}>1 \, \big| \, | Z_{0}| > x \bigr)\\[0.5em]
  &=& \lim_{x \to \infty} \frac{\Pr(Z_0 > x)}{\Pr(|Z_0| > x)}=p,
\end{eqnarray*}
and similarly $\Pr(\eta_{i0}=-1)=r$. Hence, denoting $Q_{i}=\eta_{i0}$, the limiting point process in relation (\ref{e:BaTa}) reduces to
\begin{equation}\label{e:BaTa1}
 N = \sum_{i}\delta_{(T_{i}, P_{i}Q_{i})},
\end{equation}
with $\Pr(Q_{i}=1)=p$ and $\Pr(Q_{i}=-1)=r$. Since the sequence $(Q_{i})$ is independent of the Poisson process $\sum_{i=1}^{\infty}\delta_{(T_{i}, P_{i})}$, an application of Proposition 5.2 and Proposition 5.3 in Resnick~\cite{Resnick07} yields that $N$ is a Poisson process with intensity measure $Leb \times \nu'$ where
\begin{eqnarray*}
  \nu'(\rmd x) &=& \{ \mathrm{E}[Q_{1}^{\alpha} 1_{\{ Q_{1}>0 \}}] 1_{(0,\infty)}(x) + \mathrm{E}[(-Q_{1})^{\alpha} 1_{\{ Q_{1}<0 \}}] 1_{(-\infty,0)}(x) \} \alpha
|x|^{-\alpha-1}\,\rmd x \\[0.6em]
   &=&  (p 1_{(0,\infty)}(x)+r 1_{(-\infty,0)}(x)) \alpha
|x|^{-\alpha-1}\,\rmd x\\[0.6em]
&=& \mu(\rmd x).
\end{eqnarray*}

\section{Joint functional convergence for partial sums of linear processes}
\label{S:jointconv}

Fix $0 < u < \infty$ and define the summation functional
$ \Phi^{(u)} \colon \mathbf{M}_{p}([0,1] \times \EE) \to D^{2}$
by
$$ \Phi^{(u)} \Big( \sum_{i}\delta_{(t_{i}, x_{i})} \Big) (t)
  =  \Big( \sum_{t_{i} \leq t}x_{i}\,1_{\{u < |x_{i}| < \infty \}},  \sum_{t_{i} \leq t} x_{i}^{2}\,1_{\{u < |x_{i}| < \infty \}}  \Big), \qquad t \in [0,1],$$
  where the space $\mathbf{M}_p([0,1] \times \EE)$ of Radon point
measures on $[0,1] \times \EE$ is equipped with the vague
topology (see Resnick~\cite{Re87}, Chapter 3). Let $\Lambda = \Lambda_{1} \cap \Lambda_{2}$, where
\begin{multline*}
 \Lambda_{1} =
 \{ \eta \in \mathbf{M}_{p}([0,1] \times \EE) :
   \eta ( \{0,1 \} \times \EE) = 0 = \eta ([0,1] \times \{ \pm \infty, \pm u \}) \}, \\[1em]
 \shoveleft \Lambda_{2} =
 \{ \eta \in \mathbf{M}_{p}([0,1] \times \EE) :
  \eta ( \{ t \} \times (u, \infty]) \cdot \eta ( \{ t \} \times [-\infty,-u)) = 0 \\
  \text{for all $t \in [0,1]$} \}.
\end{multline*}
The elements
of $\Lambda_2$ have the property that atoms in $[0,1] \times \{ x
\in \EE : |x| >u \} $ with the same time
coordinate are all on the same side of the time axis.
The next lemma shows that
$\Phi^{(u)}$ is continuous on the set $\Lambda$.

\begin{lem}\label{l:contfunct}
The summation functional $\Phi^{(u)} \colon \mathbf{M}_{p}([0,1]
\times \EE) \to D^{2}$ is continuous on the set $\Lambda$
when $D^{2}$ is endowed with the weak $M_{1}$ topology.
\end{lem}
\begin{proof}
For an arbitrary $\eta \in \Lambda$ suppose that $\eta_{n} \vto \eta$ in $\mathbf{M}_p([0,1] \times
\EE)$. We need to show that
$\Phi^{(u)}(\eta_n) \to \Phi^{(u)}(\eta)$ in $D^{2}$ with respect to the weak $M_1$ topology. By
Theorem 12.5.2 in Whitt~\cite{Whitt02}, it suffices to prove that,
as $n \to \infty$,
$$ d_{p}^{M_{1}}(\Phi^{(u)}(\eta_{n}), \Phi^{(u)}(\eta)) =
\max_{k=1, 2}d_{M_{1}}(\Phi^{(u)}_{k}(\eta_{n}),
\Phi^{(u)}_{k}(\eta)) \to 0.$$
Following, with small modifications, the lines in the proof of Lemma 3.2 in Basrak et al.~\cite{BKS} for each coordinate separately, we obtain
$d_{M_{1}}(\Phi^{(u)}_{1}(\eta_{n}), \Phi^{(u)}_{1}(\eta)) \to 0$ and
$d_{M_{1}}(\Phi^{(u)}_{2}(\eta_{n}), \Phi^{(u)}_{2}(\eta)) \to 0$
as
$n \to \infty$. Therefore we conclude that $\Phi^{(u)}$ is continuous at $\eta$.
\end{proof}

Let $(Z_{i})_{i \in \mathbb{Z}}$ be a strictly stationary sequence of regularly varying random variables with index $\alpha \in (0,2)$, and let
 $(C_{i})_{i \geq 0 }$ be a sequence of random variables independent of $(Z_{i})$ such that the series defying the linear process
 $$ X_{i} = \sum_{j=0}^{\infty}C_{j}Z_{i-j}, \qquad i \in \mathbb{Z},$$
  is a.s.~convergent. Let $C=\sum_{i=0}^{\infty}C_{i}$ and $C^{S}=\sum_{i=0}^{\infty}C_{i}^{2}$, and assume $C$ and $C^{S}$ are a.s.~finite. For instance, this is implied by condition (\ref{e:momcond}), since
  $$ \mathrm{E}|C|^{\delta} \leq \mathrm{E} \Big( \sum_{j=0}^{\infty} |C_{j}| \Big)^{\delta} \leq \mathrm{E} \Big( \sum_{j=0}^{\infty} |C_{j}|^{\delta} \Big) = \sum_{j=0}^{\infty} \mathrm{E} |C_{j}|^{\delta} < \infty$$
  yields $|C| < \infty$ a.s. The previous relation further implies that for almost every $\omega \in \Omega$, with $(\Omega, \mathcal{F}, \Pr)$ being the underlying probability space, it holds that $\sum_{j=0}^{\infty}|C_{j}(\omega)| < \infty$. Hence there exists an integer $j_{0}=j_{0}(\omega) \geq 0$ such that $|C_{j}(\omega)| \leq 1$ for all $j \geq j_{0}$. This leads to $C_{j}^{2}(\omega) \leq |C_{j}(\omega)|$ for $j \geq j_{0}(\omega)$, which implies
 $$ \sum_{j=j_{0}}^{\infty} C_{j}^{2}(\omega) \leq \sum_{j=j_{0}}^{\infty} |C_{j}(\omega)| < \infty,$$
 i.e.~$C^{S}(\omega) < \infty$.
  For $n \in \mathbb{N}$ define
$$ L_{n}(t) =  \bigg( \sum_{k=1}^{[nt]} \frac{X_{k}}{a_{n}}, \sum_{k=1}^{\lfloor nt \rfloor}\frac{X_{k}^{2}}{a_{n}^{2}} \bigg),
   \qquad t \in [0,1],$$
with $a_{n}$ as in (\ref{eq:niz}). In the proposition below we derive functional convergence of a particular stochastic process $\widetilde{L}_{n}$, constructed from the sequence $(Z_{i})$, in the space $D^{2}$ with the weak $M_{1}$ topology. The next step will be to show that the $M_{2}$ distance between processes $L_{n}$ and $\widetilde{L}_{n}$ is asymptotically negligible as $n$ tends to infinity.

For the next result in the case $\alpha \in [1,2)$ we will need to assume the following condition to deal with small jumps:
 \begin{equation}\label{e:vsvcond}
 \lim_{u \downarrow 0} \limsup_{n \to \infty}~\Pr \bigg[ \max_{1 \leq k \leq n} \bigg| \sum_{i=1}^{k}  \bigg( \frac{Z_{i}}{a_{n}} 1_{\big\{ \frac{|Z_{i}|}{a_{n}} \leq u\big\}} - \mathrm{E} \bigg( \frac{Z_{i}}{a_{n}} 1_{\big\{ \frac{|Z_{i}|}{a_{n}} \leq u \big\}} \bigg) \bigg) \bigg| > \epsilon \bigg]=0
 \end{equation}
for all $\epsilon >0$. This condition holds if the sequence $(Z_{i})$ is $\rho$-mixing at a certain rate (see Lemma 4.8 in Tyran-Kami\'{n}ska~\cite{Ty10a}), in particular it holds for i.i.d.~sequences.

\begin{prop}\label{p:FLT}
 Let $(Z_{i})_{i \in \mathbb{Z}}$ be a strictly stationary and strongly mixing sequence of regularly varying random variables satisfying $(\ref{e:regvar})$ and $(\ref{eq:pq})$ with $\alpha \in (0,2)$, such that the local dependence condition $D'$ and conditions $(\ref{e:oceknula})$ and $(\ref{e:sim})$ hold. If $\alpha \in [1,2)$, also suppose that condition $(\ref{e:vsvcond})$ holds. Let
 $(C_{i})_{i \geq 0 }$ be a sequence of random variables independent of $(Z_{i})$ such that the series defying the linear process
 $$ X_{i} = \sum_{j=0}^{\infty}C_{j}Z_{i-j}, \qquad i \in \mathbb{Z},$$
  is a.s.~convergent.
Then, as $n \to \infty$,
\begin{equation*}\label{e:pomkonv}
\widetilde{L}_{n}(\,\cdot\,) := \bigg( \sum_{i = 1}^{\lfloor n \, \cdot \, \rfloor} \frac{CZ_{i}}{a_{n}}, \sum_{i = 1}^{\lfloor n \, \cdot \, \rfloor} \frac{C^{S} Z_{i}^{2}}{a_{n}^{2}}  \bigg) \dto  (C^{(1)}V(\,\cdot\,), C^{(2)}W(\,\cdot\,))
\end{equation*}
 in $D^{2}$ with the weak $M_{1}$ topology,
where $V$ is an $\alpha$--stable L\'{e}vy process with characteristic triple $(0, \mu, b)$, with $\mu$ as in $(\ref{eq:mu})$,
$$ b = \left\{ \begin{array}{cc}
                                   0, & \quad \alpha = 1,\\[0.4em]
                                   (p-r)\frac{\alpha}{1-\alpha}, & \quad \alpha \in (0,1) \cup (1,2),
                                 \end{array}\right.$$
with $p$ and $r$ defined in $(\ref{eq:pq})$, $W$ is an $\alpha/2$--stable L\'{e}vy process with characteristic triple $(0, \widetilde{\mu}, \alpha/(2-\alpha))$ with $\widetilde{\mu}(\rmd x) = (\alpha /2) x^{-\alpha/2-1} 1_{(0,\infty)}(x)\,\rmd x$,
and $(C^{(1)}, C^{(2)})$ is a random vector, independent of $(V, W)$, such that
$(C^{(1)}, C^{(2)}) \eind (C, C^{S})$.
\end{prop}

Before the proof of the proposition recall that the distribution of a L\'{e}vy process $V$ is characterized by its
characteristic triple (i.e.~the characteristic triple of the infinitely divisible distribution
of $V(1)$). The characteristic function of $V(1)$ and the characteristic triple
$(a, \rho, c)$ are related in the following way:
 $$
  \mathrm{E} [e^{izV(1)}] = \exp \biggl( -\frac{1}{2}az^{2} + icz + \int_{\mathbb{R}} \bigl( e^{izx}-1-izx 1_{[-1,1]}(x) \bigr)\,\rho(\rmd x) \biggr)$$
for $z \in \mathbb{R}$, where $a \ge 0$, $c \in \mathbb{R}$ are constants, and $\rho$ is a measure on $\mathbb{R}$ satisfying
$$ \rho ( \{0\})=0 \qquad \text{and} \qquad \int_{\mathbb{R}}(|x|^{2} \wedge 1)\,\rho(\rmd x) < \infty,$$
where $a \wedge b = \min \{a,b\}$.
 We refer to Sato~\cite{Sa99} for a textbook treatment of
L\'{e}vy processes.

\begin{proof}[Proof of Proposition~\ref{p:FLT}]
Take an arbitrary $0<u<1$, and consider
$$ \Phi^{(u)}(N_{n})(\,\cdot\,) = \bigg( \sum_{i/n \leq\,\cdot}\frac{Z_{i}}{a_{n}} 1_{ \big\{ \frac{|Z_{i}|}{a_{n}} > u \big\} }, \sum_{i/n \leq\,\cdot}\frac{Z_{i}^{2}}{a_{n}^{2}} 1_{ \big\{ \frac{|Z_{i}|}{a_{n}} > u \big\} } \bigg).$$
The limiting point process $N$ in relations (\ref{e:BaTa}) and (\ref{e:BaTa1}) is a Poisson process, and hence it is almost surely contained in the set $\Lambda$ (see Resnick~\cite{Resnick07}, page 221). Therefore, since by Lemma~\ref{l:contfunct} the functional $\Phi^{(u)}$ is continuous on the set $\Lambda$, the continuous mapping theorem applied to the convergence in (\ref{e:BaTa}) yields $\Phi^{(u)}(N_{n}) \dto \Phi^{(u)}(N)$ in $D^{2}$ under the weak $M_{1}$ topology, i.e.
\begin{equation}\label{e:mainconv0}
 \bigg( \sum_{i = 1}^{\lfloor n \, \cdot \, \rfloor} \frac{Z_{i}}{a_{n}}
    1_{ \bigl\{ \frac{|Z_{i}|}{a_{n}} > u \bigr\} },  \sum_{i = 1}^{\lfloor n \, \cdot \, \rfloor} \frac{Z_{i}^{2}}{a_{n}^{2}}
    1_{ \bigl\{ \frac{|Z_{i}|}{a_{n}} > u \bigr\} } \bigg)
    \dto
   \bigg( \sum_{T_{i} \le \, \cdot} P_{i}Q_{i} 1_{\{ |P_{i}Q_{i}| > u \}}, \sum_{T_{i} \le \, \cdot} P_{i}^{2}Q_{i}^{2} 1_{\{ |P_{i}Q_{i}| > u \}} \bigg).
 \end{equation}
 Since $P_{i}>0$ and $|Q_{i}|=1$ for all $i$, the limiting process in (\ref{e:mainconv0}) reduces to
 $$  \bigg( \sum_{T_{i} \le \, \cdot} P_{i}Q_{i} 1_{\{ P_{i} > u \}}, \sum_{T_{i} \le \, \cdot} P_{i}^{2} 1_{\{P_{i} > u \}} \bigg).$$
Relation (\ref{eq:onedimregvar}) implies, as $n \to \infty$,
\begin{eqnarray}\label{e:conv12}
 \nonumber \floor{nt} \mathrm{E} \Big( \frac{Z_{1}}{a_{n}} 1_{ \big\{ u < \frac{|Z_{1}|}{a_{n}} \leq 1 \big\} } \Big) &=& \frac{\floor{nt}}{n} \int_{u < |x| \leq 1} x\,n\Pr \Big( \frac{Z_{1}}{a_{n}} \in \rmd x \Big) \\[0.6em]
   & \to &  t \int_{u < |x| \leq 1} x\mu(\rmd x)
\end{eqnarray}
for every $t \in [0,1]$, and similarly
\begin{equation}\label{e:conv122}
 \floor{nt} \mathrm{E} \Big( \frac{Z_{1}^{2}}{a_{n}^{2}} 1_{ \big\{ u < \frac{|Z_{1}|}{a_{n}} \leq 1 \big\} } \Big)
    \to  t \int_{u < |x| \leq 1} x^{2}\mu(\rmd x),
\end{equation}
and these convergences are uniform in $t$.
Therefore an application of Lemma~\ref{l:weakM1transf} to (\ref{e:mainconv0}), (\ref{e:conv12}) and (\ref{e:conv122})  yields, as $n \to \infty$,
\begin{multline}\label{e:mainconv}
      L_{n}^{(u)}(\,\cdot\,) := \bigg( \sum_{i = 1}^{\lfloor n \, \cdot \, \rfloor} \frac{Z_{i}}{a_{n}}
    1_{ \bigl\{ \frac{|Z_{i}|}{a_{n}} > u \bigr\} } - \floor{n\,\cdot\,}b_{n}^{(u)}, \sum_{i = 1}^{\lfloor n \, \cdot \, \rfloor} \frac{Z_{i}^{2}}{a_{n}^{2}}
    1_{ \bigl\{ \frac{|Z_{i}|}{a_{n}} > u \bigr\} } - \floor{n\,\cdot\,}c_{n}^{(u)} \bigg) \\
    \dto L^{(u)}(\,\cdot\,) :=  \bigg( \sum_{T_{i} \le \, \cdot} P_{i}Q_{i} 1_{\{ P_{i} > u \}} - (\,\cdot\,) b^{(u)}, \sum_{T_{i} \le \, \cdot} P_{i}^{2} 1_{\{ P_{i} > u \}} - (\,\cdot\,) c^{(u)} \bigg)
\end{multline}
in $D^{2}$ under the weak $M_{1}$ topology, where
$$ b_{n}^{(u)} = \mathrm{E} \Big( \frac{Z_{1}}{a_{n}} 1_{ \big\{ u < \frac{|Z_{1}|}{a_{n}} \leq 1 \big\} } \Big), \qquad c_{n}^{(u)} = \mathrm{E} \Big( \frac{Z_{1}^{2}}{a_{n}^{2}} 1_{ \big\{ u < \frac{|Z_{1}|}{a_{n}} \leq 1 \big\} } \Big)$$
and
$$ b^{(u)} = \int_{u < |x| \leq 1} x\mu(\rmd x), \qquad c^{(u)} = \int_{u < |x| \leq 1} x^{2}\mu(\rmd x). $$
Since $N=\sum_{i} \delta_{(T_{i}, P_{i}Q_{i} })$
 is a Poisson process with intensity measure
$ Leb \times \mu$, by the It\^{o} representation of a L\'{e}vy process (see Resnick~\cite{Resnick07}, pages 150--157; and Sato~\cite{Sa99}, Theorem 14.3 and Theorem 19.2), there exists an $\alpha$--stable L\'{e}vy process $V^{(0)}(\,\cdot\,)$ with characteristic triple $(0, \mu, 0)$ such that
$$ \sup_{t \in [0,1]} |L^{(u)1}(t)-V^{(0)}(t)| \to 0$$
almost surely as $u \to 0$, where $L^{(u)1}$ is the first component of the process $L^{(u)}$. Since uniform convergence implies Skorokhod $M_{1}$ convergence, we get
\begin{equation}\label{e:dm1}
d_{M_{1}} ( L^{(u)1}, V^{(0)} ) \to 0
\end{equation}
almost surely as $u \to 0$. Proposition 5.2 in Resnick~\cite{Resnick07} implies that $\sum_{i} \delta_{(T_{i}, P_{i}^{2})}$
 is a Poisson process with intensity measure
$ Leb \times \widetilde{\mu}$, with $\widetilde{\mu}(\rmd x) = (\alpha /2) x^{-\alpha/2-1}\,\rmd x$ for $x>0$.
A new application of the It\^{o} representation yields the existence of an $\alpha/2$--stable L\'{e}vy process $W^{(0)}(\,\cdot\,)$ with characteristic triple $(0, \widetilde{\mu}, 0)$ such that
\begin{equation}\label{e:dm12}
d_{M_{1}} ( L^{(u)2}, W^{(0)} ) \to 0
\end{equation}
almost surely as $u \to 0$, where $L^{(u)2}$ is the second component of the process $L^{(u)}$.
Let
$L^{(0)} := (V^{(0)}, W^{(0)})$.
Then from relations (\ref{e:dm1}) and (\ref{e:dm12}), and the definition of the metric $d_{p}^{M_{1}}$ in (\ref{e:defdpM1}) we obtain
\begin{equation*}
d_{p}^{M_{1}}(L^{(u)}, L^{(0)}) \to 0
\end{equation*}
almost surely as $u \to 0$. Since almost sure convergence implies convergence in distribution, we have, as $u \to 0$,
\begin{equation}\label{e:mainconv2}
 L^{(u)}(\,\cdot\,) \dto L^{(0)}(\,\cdot\,)
\end{equation}
in $D^{2}$ endowed with the weak $M_{1}$ topology.
Let
$$   L_{n}^{(0)}(\,\cdot\,) := \bigg( \sum_{i = 1}^{\lfloor n \, \cdot \, \rfloor} \frac{Z_{i}}{a_{n}}
     - \floor{n\,\cdot\,} \mathrm{E} \Big( \frac{Z_{1}}{a_{n}} 1_{ \big\{ \frac{|Z_{1}|}{a_{n}} \leq 1 \big\} } \Big) , \sum_{i = 1}^{\lfloor n \, \cdot \, \rfloor} \frac{Z_{i}^{2}}{a_{n}^{2}}
     - \floor{n\,\cdot\,} \mathrm{E} \Big( \frac{Z_{1}^{2}}{a_{n}^{2}} 1_{ \big\{ \frac{|Z_{1}|}{a_{n}} \leq 1 \big\} } \Big)\bigg).
    $$
If we show that
$$ \lim_{u \to 0}\limsup_{n \to \infty} \Pr(d_{p}^{M_{1}}(L_{n}^{(0)},L_{n}^{(u)}) > \epsilon)=0$$
for every $\epsilon >0$, from (\ref{e:mainconv}) and (\ref{e:mainconv2}) by a generalization of Slutsky's theorem (see Theorem 3.5 in Resnick~\cite{Resnick07}) it will follow that
$ L_{n}^{(0)} \dto L^{(0)}$ as $n \to \infty$,
in $D^{2}$ with the weak $M_{1}$ topology.
Recalling the definitions, the fact that the metric $d_{p}^{M_{1}}$ is bounded above by the uniform metric (see Theorem 12.10.3 in Whitt~\cite{Whitt02}), and using stationarity and Markov's inequality we obtain
 \begin{eqnarray}\label{e:slutsky}
    \nonumber  \Pr ( d_{p}^{M_{1}}(L_{n}^{(0)}, L_{n}^{(u)}) > \epsilon ) & \leq & \Pr \bigg(
     \sup_{t \in [0,1]} \|L_{n}^{(0)}(t) - L_{n}^{(u)}(t)\| >  \epsilon \bigg) \\[0.2em]
   \nonumber & \hspace*{-16em} \leq & \ \hspace*{-8em} \Pr \bigg[
       \sup_{t \in [0,1]}  \bigg| \sum_{i=1}^{\lfloor nt \rfloor} \frac{Z_{i}}{a_{n}}
       1_{ \big\{ \frac{|Z_{i}|}{a_{n}} \leq u \big\} } - \lfloor nt \, \rfloor \mathrm{E} \Big( \frac{Z_{1}}{a_{n}} 1_{ \big\{ \frac{|Z_{1}|}{a_{n}} \leq u \big\} } \Big) \bigg|  > \epsilon
       \bigg]\\[0.3em]
   \nonumber & \hspace*{-16em} & \ \hspace*{-8em} + \Pr \bigg[
       \sup_{t \in [0,1]}  \bigg| \sum_{i=1}^{\lfloor nt \rfloor} \frac{Z_{i}^{2}}{a_{n}^{2}}
       1_{ \big\{ \frac{|Z_{i}|}{a_{n}} \leq u \big\} } - \lfloor nt \, \rfloor \mathrm{E} \Big( \frac{Z_{1}^{2}}{a_{n}^{2}} 1_{ \big\{ \frac{|Z_{1}|}{a_{n}} \leq u \big\} } \Big) \bigg|  > \epsilon
       \bigg]\\[0.3em]
    \nonumber & \hspace*{-16em} = & \ \hspace*{-8em} \Pr \bigg[
       \max_{1 \leq k \leq n}  \bigg| \sum_{i=1}^{k} \bigg( \frac{Z_{i}}{a_{n}}
       1_{ \big\{ \frac{|Z_{i}|}{a_{n}} \leq u \big\} } - \mathrm{E} \Big( \frac{Z_{1}}{a_{n}} 1_{ \big\{ \frac{|Z_{1}|}{a_{n}} \leq u \big\} } \Big) \bigg) \bigg|  > \epsilon
       \bigg]\\[0.3em]
    \nonumber & \hspace*{-16em} & \ \hspace*{-8em} + \Pr \bigg[
       \max_{1 \leq k \leq n}  \bigg| \sum_{i=1}^{k} \bigg( \frac{Z_{i}^{2}}{a_{n}^{2}}
       1_{ \big\{ \frac{|Z_{i}|}{a_{n}} \leq u \big\} } - \mathrm{E} \Big( \frac{Z_{1}^{2}}{a_{n}^{2}} 1_{ \big\{ \frac{|Z_{1}|}{a_{n}} \leq u \big\} } \Big) \bigg) \bigg|  > \epsilon
       \bigg]
 \end{eqnarray}
Therefore we have to show
\begin{equation}\label{e:slutskycondition}
 \lim_{u \to 0}\limsup_{n \to \infty} \Pr \bigg[
       \max_{1 \leq k \leq n}  \bigg| \sum_{i=1}^{k} \bigg( \frac{Z_{i}}{a_{n}}
       1_{ \big\{ \frac{|Z_{i}|}{a_{n}} \leq u \big\} } - \mathrm{E} \Big( \frac{Z_{1}}{a_{n}} 1_{ \big\{ \frac{|Z_{1}|}{a_{n}} \leq u \big\} } \Big) \bigg) \bigg|  > \epsilon
       \bigg]
\end{equation}
and
\begin{equation}\label{e:slutskycondition2}
 \lim_{u \to 0}\limsup_{n \to \infty} \Pr \bigg[
       \max_{1 \leq k \leq n}  \bigg| \sum_{i=1}^{k} \bigg( \frac{Z_{i}^{2}}{a_{n}^{2}}
       1_{ \big\{ \frac{|Z_{i}|}{a_{n}} \leq u \big\} } - \mathrm{E} \Big( \frac{Z_{1}^{2}}{a_{n}^{2}} 1_{ \big\{ \frac{|Z_{1}|}{a_{n}} \leq u \big\} } \Big) \bigg) \bigg|  > \epsilon
       \bigg]
\end{equation}
For $\alpha \in [1,2)$ relation (\ref{e:slutskycondition}) is simply condition (\ref{e:vsvcond}). In the case $\alpha \in (0,1)$, let
$$ I(u,n,\epsilon) = \Pr \bigg[
       \max_{1 \le k \le n}  \bigg| \sum_{i=1}^{k} \bigg( \frac{Z_{i}}{a_{n}} \,
       1_{ \big\{ \frac{|Z_{i}|}{a_{n}} \le u \big\} } -  \E \bigg( \frac{Z_{1}}{a_{n}} \,
       1_{ \big\{ \frac{|Z_{1}|}{a_{n}} \le u \big\} } \bigg) \bigg) \bigg| > \epsilon
       \bigg].$$
 Using stationarity and Markov's inequality we get the bound
 \begin{eqnarray}\label{e:alpha01}
   \nonumber I(u,n,\epsilon) & \le & \epsilon^{-1} \E \bigg[ \sum_{i=1}^{n} \bigg| \frac{Z_{i}}{a_{n}} \,
       1_{ \big\{ \frac{|Z_{i}|}{a_{n}} \le u \big\} } -  \E \bigg( \frac{Z_{1}}{a_{n}} \,
       1_{ \big\{ \frac{|Z_{1}|}{a_{n}} \le u \big\} } \bigg) \bigg| \bigg]\\[0.4em]
     \nonumber    &\leq& \frac{2n}{\epsilon} \E \bigg( \frac{|Z_{1}|}{a_{n}} \, 1_{ \big\{ \frac{|Z_{1}|}{a_{n}}
            \le u \big\} } \bigg)\\[0.4em]
  \nonumber & = & \ \frac{2u}{\epsilon} \cdot n \Pr (|Z_{1}| > a_{n}) \cdot \frac{\Pr(|Z_{1}| > ua_{n})}{\Pr(|Z_{1}|>a_{n})} \cdot
          \frac{\mathrm{E}(|Z_{1}| 1_{ \{ |Z_{1}| \leq ua_{n} \} })}{ ua_{n} \Pr (|Z_{1}| >ua_{n})}.
 \end{eqnarray}
 Since the random variable $Z_{1}$ is regularly varying with index $\alpha$, it holds that
  $$ \lim_{n \to \infty} \frac{\Pr(|Z_{1}| > ua_{n})}{\Pr(|Z_{1}|>a_{n})} = u^{-\alpha}.$$
  By Karamata's theorem
  $$ \lim_{n \to \infty} \frac{\mathrm{E}(|Z_{1}| 1_{ \{ |Z_{1}| \leq ua_{n} \} })}{ ua_{n} \Pr (|Z_{1}| >ua_{n})} = \frac{\alpha}{1-\alpha},$$
 and therefore taking into account relation (\ref{eq:niz}), we get
 $$ \limsup_{n \to \infty}  I(u,n,\epsilon) \leq u^{1-\alpha} \frac{2 \alpha}{\epsilon (1-\alpha)}.$$
 Since in this case $1-\alpha >0$, letting $u \to 0$ we obtain
 $$ \lim_{u \to 0}\limsup_{n \to \infty} I(u, n, \epsilon)=0.$$
 Now we prove relation (\ref{e:slutskycondition2}). Let
$$ J(u,n,\epsilon) = \Pr \bigg[
       \max_{1 \le k \le n}  \bigg| \sum_{i=1}^{k} \bigg( \frac{Z_{i}^{2}}{a_{n}^{2}} \,
       1_{ \big\{ \frac{|Z_{i}|}{a_{n}} \le u \big\} } -  \E \bigg( \frac{Z_{1}^{2}}{a_{n}^{2}} \,
       1_{ \big\{ \frac{|Z_{1}|}{a_{n}} \le u \big\} } \bigg) \bigg) \bigg| > \epsilon
       \bigg],$$
and similarly as for $I(u,n, \epsilon)$ we derive
 \begin{eqnarray*}
   \nonumber J(u,n,\epsilon) & \le & \epsilon^{-1} \E \bigg[ \sum_{i=1}^{n} \bigg| \frac{Z_{i}^{2}}{a_{n}^{2}} \,
       1_{ \big\{ \frac{|Z_{i}|}{a_{n}} \le u \big\} } -  \E \bigg( \frac{Z_{1}^{2}}{a_{n}^{2}} \,
       1_{ \big\{ \frac{|Z_{1}|}{a_{n}} \le u \big\} } \bigg) \bigg| \bigg]\\[0.4em]
  \nonumber & \leq & \ \frac{2u^{2}}{\epsilon} \cdot n \Pr (|Z_{1}| > a_{n}) \cdot \frac{\Pr(|Z_{1}| > ua_{n})}{\Pr(|Z_{1}|>a_{n})} \cdot
          \frac{\mathrm{E}(Z_{1}^{2} 1_{ \{ |Z_{1}| \leq ua_{n} \} })}{ u^{2}a_{n}^{2} \Pr (|Z_{1}| >ua_{n})}.
 \end{eqnarray*}
  By Karamata's theorem
  $$ \lim_{n \to \infty} \frac{\mathrm{E}(Z_{1}^{2} 1_{ \{ |Z_{1}| \leq ua_{n} \} })}{ u^{2}a_{n}^{2} \Pr (|Z_{1}| >ua_{n})} = \frac{\alpha}{2-\alpha},$$
 and hence
 $$ \limsup_{n \to \infty}  J(u,n,\epsilon) \leq u^{2-\alpha} \frac{2 \alpha}{\epsilon (2-\alpha)}.$$
 Since $\alpha \in (0,2)$, letting $u \to 0$ we obtain
 $$ \lim_{u \to 0}\limsup_{n \to \infty} J(u, n, \epsilon)=0.$$
 Therefore we conclude
 \begin{equation}\label{e:weakM1L0}
  L_{n}^{(0)}(\,\cdot\,) \dto L^{(0)}(\,\cdot\,) \qquad \textrm{as} \ n \to \infty,
  \end{equation}
in $D^{2}$ with the weak $M_{1}$ topology.
By Karamata's theorem, as $n \to \infty$,
\begin{eqnarray*}
  n\,\mathrm{E} \Big( \frac{Z_{1}}{a_{n}} 1_{\{|Z_{1}| \leq a_{n} \}} \Big) \to (p-r)\frac{\alpha}{1-\alpha}, && \textrm{if}  \ \ \alpha \in (0,1),\\[0.5em]
  n\,\mathrm{E} \Big( \frac{Z_{1}}{a_{n}} 1_{\{ |Z_{1}| > a_{n} \}} \Big) \to (p-r)\frac{\alpha}{\alpha-1}, && \textrm{if} \ \ \alpha \in (1,2),
\end{eqnarray*}
with $p$ and $r$ as in (\ref{eq:pq}), and similarly
\begin{equation*}
  n\,\mathrm{E} \Big( \frac{Z_{1}^{2}}{a_{n}^{2}} 1_{\{|Z_{1}| \leq a_{n} \}} \Big) \to \frac{\alpha}{2-\alpha}.
\end{equation*}
Therefore conditions (\ref{e:oceknula}) and (\ref{e:sim}), and Lemma~\ref{l:weakM1transf}, applied to the convergence in (\ref{e:weakM1L0}), yield
\begin{equation}\label{eq:weakconvLn*}
 L_{n}^{*}(\,\cdot\,) \dto L(\,\cdot\,) \qquad \textrm{as} \  n \to \infty
\end{equation}
in $D^{2}$ with the weak $M_{1}$ topology, where
$$L_{n}^{*}(\,\cdot\,) := \bigg( \sum_{i = 1}^{\lfloor n \, \cdot \, \rfloor} \frac{Z_{i}}{a_{n}}, \sum_{i = 1}^{\lfloor n \, \cdot \, \rfloor} \frac{Z_{i}^{2}}{a_{n}^{2}}  \bigg)$$
and $L:=(V, W)$, with
$$V(t)= \left\{ \begin{array}{cc}
                                   V^{(0)}(t), & \quad \alpha = 1,\\[0.4em]
                                   V^{(0)}(t) + t(p-r)\frac{\alpha}{1-\alpha}, & \quad \alpha \in (0,1) \cup (1,2),
                                 \end{array}\right.$$
being an $\alpha$--stable L\'{e}vy process
 with characteristic triple
$(0,\mu,0)$ if $\alpha=1$ and $(0,\mu,(p-r)\alpha/(1-\alpha))$ if $\alpha \in (0,1) \cup (1,2)$, and
$$W(t) = W^{(0)}(t) + t \frac{\alpha}{2-\alpha}$$
 being an $\alpha/2$--stable L\'{e}vy process with characteristic triple $(0,\widetilde{\mu}, \alpha/(2-\alpha))$.

 It is known that $D^{1}$ equipped with the Skorokhod $J_{1}$ topology is a Polish space, i.e.~metrizable as a complete separable metric space (see Billingsley~\cite{Bi68}, Section 14), and therefore the same holds for the $M_{1}$ topology, since it is topologically complete (see Whitte~\cite{Whitt02}, Section 12.8) and separability remains preserved in the weaker topology.
 The space $D^{2}$ equipped with the weak $M_{1}$ topology is separable as a direct product of two separable topological spaces. It is also topologically complete since the product metric $d_{p}^{M_{1}}$ inherits the completeness of the component metrics. Therefore $D^{2}$ with the weak $M_{1}$ topology is also a Polish space, and this suffices to conclude, by Corollary 5.18 in Kallenberg~\cite{Ka97}, that there exists a random vector ($C^{(1)}, C^{(2)})$, independent of $(V, W)$, such that
\begin{equation}\label{e:eqdistrC}
(C^{(1)}, C^{(2)}) \eind (C, C^{S}).
\end{equation}
This, relation (\ref{eq:weakconvLn*}), the fact that $(C, C^{S})$ is independent of $L_{n}^{*}$, and Theorem 3.29 in Kallenberg~\cite{Ka97}, imply that, as $n \to \infty$,
  \begin{equation}\label{e:zajedkonvK1}
   (B, B^{S}, L_{n}^{*1}, L_{n}^{*2}) \dto (B^{(1)}, B^{(2)}, V, W)
  \end{equation}
  in $D^{4}$ with the product $M_{1}$ topology, where $L_{n}^{*i}$ is the $i$--th component of $L_{n}^{*}$ $(i=1,2)$, $B(t)=C$, $B^{S}(t)=C^{S}$, $B^{(1)}(t)=C^{(1)}$ and $B^{(2)}(t)=C^{(2)}$ for $t \in [0,1]$.

Let $g \colon D^{4} \to D^{2}$ be a function defined by
$$ g(x) = (x_{1} \cdot x_{3}, x_{2} \cdot x_{4}), \qquad x=(x_{1}, \ldots, x_{4}) \in D^{4}.$$
 Denote by $\widetilde{D}_{1,2}$ the set of all functions in $D^{4}$ for which the first two component functions have no discontinuity points, that is
$$ \widetilde{D}_{1,2} = \{ (x_{1}, \ldots, x_{4}) \in D^{4} : \textrm{Disc}(x_{i}) = \emptyset, \ i=1,2 \}.$$
By Lemma~\ref{l:contmultpl} the function $g$ is continuous on the set $\widetilde{D}_{1,2}$ in the weak $M_{1}$ topology, and hence $\textrm{Disc}(g) \subseteq \widetilde{D}_{1,2}^{c}$. Denoting
$\widetilde{D}_{1} = \{ u \in D^{1} : \textrm{Disc}(u)= \emptyset \}$ we obtain
\begin{eqnarray*}
\Pr[ (B^{(1)}, B^{(2)}, V, W) \in \textrm{Disc}(g) ] & \leq & \Pr [ (B^{(1)}, B^{(2)}, V, W)  \in \widetilde{D}_{1,2}^{c}]\\[0.6em]
 &   \leq &  \Pr [ \{ B^{(1)} \in \widetilde{D}_{1}^{c} \} \cup \{ B^{(2)} \in \widetilde{D}_{1}^{c} \}]=0,
 \end{eqnarray*}
where the last equality holds since $B^{(1)}$ and $B^{(2)}$ have no discontinuity points. This allows us to apply the continuous mapping theorem to relation (\ref{e:zajedkonvK1}) yielding
$g(B, B^{S}, L_{n}^{*1}, L_{n}^{*2}) \dto g(B^{(1)}, B^{(2)}, V, W)$ as $n \to \infty$, that is
\begin{eqnarray}\label{e:zajedkonvK2}
 \nonumber \bigg( \sum_{i = 1}^{\lfloor n \, \cdot \, \rfloor} \frac{CZ_{i}}{a_{n}}, \sum_{i = 1}^{\lfloor n \, \cdot \, \rfloor} \frac{C^{S} Z_{i}^{2}}{a_{n}^{2}}  \bigg)
 & \dto & (C^{(1)} V(\,\cdot\,), C^{(2)} W(\,\cdot\,))
\end{eqnarray}
 in $D^{2}$ with the weak $M_{1}$ topology, which completes the proof.
\end{proof}

\section{Finite order linear processes}
\label{S:FiniteMA}

Fix $q \in \mathbb{N}$ and let $C_{0}, C_{1}, \ldots , C_{q}$ be random variables satisfying
\be\label{eq:FiniteMAcond}
0 \le \sum_{i=0}^{s}C_{i} \Bigg/ \sum_{i=0}^{q}C_{i} \le 1 \ \ \textrm{a.s.} \qquad \textrm{for every} \ s=0, 1, \ldots, q.
\ee
Let
$C= \sum_{i=0}^{q}C_{i}$, and assume $C \neq 0$ a.s. If the $C_{j}$'s are all nonnegative or all nonpositive, then condition (\ref{eq:FiniteMAcond}) is trivially satisfied. Let $L_{n}=(V_{n}, W_{n})$, where
\begin{equation*}
V_{n}(t) = \frac{1}{a_{n}}  \sum_{i=1}^{\floor {nt}}X_{i} \qquad \textrm{and} \qquad W_{n}(t)= \frac{1}{a_{n}^{2}}  \sum_{i=1}^{\floor {nt}}X_{i}^{2} \qquad (t \in [0,1])
\end{equation*}
are the partial sum processes constructed from linear processes
$$ X_{i} = \sum_{j=0}^{q}C_{j}Z_{i-j}, \qquad i \in \mathbb{Z},$$
with the normalizing sequence $(a_n)$ as in~\eqref{eq:niz}.

\begin{thm}\label{t:FinMA}
Let $(Z_{i})_{i \in \mathbb{Z}}$ be an i.i.d.~sequence of regularly varying random variables satisfying $(\ref{e:regvar})$ and $(\ref{eq:pq})$ with $\alpha \in (0,2)$. Assume conditions $(\ref{e:oceknula})$ and $(\ref{e:sim})$ hold. Let
 $C_{0}, \ldots, C_{q}$ be random variables, independent of $(Z_{i})$, such that condition $(\ref{eq:FiniteMAcond})$ holds.
Then, as $n \to \infty$,
\begin{equation*}\label{e:mainconvFinMA}
L_{n}(\,\cdot\,) \dto  (C^{(1)}V(\,\cdot\,), C^{(2)}W(\,\cdot\,))
\end{equation*}
 in $D^{2}$ with the weak $M_{2}$ topology,
where $V$ is an $\alpha$--stable L\'{e}vy process with characteristic triple $(0, \mu, b)$, with $\mu$ as in $(\ref{eq:mu})$,
$$ b = \left\{ \begin{array}{cc}
                                   0, & \quad \alpha = 1,\\[0.4em]
                                   (p-r)\frac{\alpha}{1-\alpha}, & \quad \alpha \in (0,1) \cup (1,2),
                                 \end{array}\right.$$
with $p$ and $r$ defined in $(\ref{eq:pq})$, $W$ is an $\alpha/2$--stable L\'{e}vy process with characteristic triple $(0, \widetilde{\mu}, \alpha/(2-\alpha))$ with $\widetilde{\mu}(\rmd x) = (\alpha /2) x^{-\alpha/2-1} 1_{(0,\infty)}(x)\,\rmd x$,
and $(C^{(1)}, C^{(2)})$ is a random vector, independent of $(V, W)$, such that
$(C^{(1)}, C^{(2)}) \eind (C, C^{S})$.
\end{thm}

\begin{proof} 
By Proposition~\ref{p:FLT}, $\widetilde{L}_{n} = (\widetilde{V}_{n}, \widetilde{W}_{n}) \dto  (C^{(1)}V, C^{(2)}W)$ as $n \to \infty$, in $D^{2}$ with the weak $M_{1}$ topology, where
$$ \widetilde{V}_{n}(t) =  \sum_{i=1}^{\floor {nt}}\frac{CZ_{i}}{a_{n}} \qquad \textrm{and} \qquad \widetilde{W}_{n}(t)= \sum_{i=1}^{\floor {nt}}\frac{C^{S}Z_{i}^{2}}{a_{n}^{2}}, \qquad t \in [0,1].$$
 Since $M_{1}$ convergence implies $M_{2}$ convergence, we have
\begin{equation}\label{e:pomkonvFMA}
\widetilde{L}_{n}(\,\cdot\,) = (\widetilde{V}_{n}(\,\cdot\,), \widetilde{W}_{n}(\,\cdot\,)) \dto  (C^{(1)}V(\,\cdot\,), C^{(2)}W(\,\cdot\,))
\end{equation}
 in $D^{2}$ with the weak $M_{2}$ topology as well.

If we show that
$$ \lim_{n \to \infty} \Pr [ d_{p}^{M_{2}}(\widetilde{L}_{n}, L_{n}) > \delta ]=0$$
for any $\delta > 0$, then from (\ref{e:pomkonvFMA}) by an application of Slutsky's theorem (see Theorem 3.4 in Resnick~\cite{Resnick07}) it will follow that $L_{n} \dto (C^{(1)}V, C^{(2)}W)$ as $n \to \infty$, in $D^{2}$ with the weak $M_{2}$ topology. By the definition of the metric $d_{p}^{M_{2}}$ in (\ref{e:defdp}) it suffices to show that
\begin{equation}\label{e:SlutskyV}
\lim_{n \to \infty} \Pr [ d_{M_{2}}(\widetilde{V}_{n}, V_{n}) > \delta ]=0
\end{equation}
and
\begin{equation}\label{e:SlutskyM}
\lim_{n \to \infty} \Pr [ d_{M_{2}}(\widetilde{W}_{n}, W_{n}) > \delta ]=0.
\end{equation}
Relation (\ref{e:SlutskyV}) was established in the proof of Theorem 2.1 in Krizmani\'{c}~\cite{Kr19} under condition (\ref{eq:FiniteMAcond}).
Now we turn our attention to relation (\ref{e:SlutskyM}). Fix $\delta >0$ and let $n \in \mathbb{N}$ be large enough, i.e.~$n > \max\{2q, 2q/\delta \}$.
By the definition of the metric $d_{M_{2}}$ in (\ref{e:defM2H}) we have
\begin{eqnarray*}
  d_{M_{2}}(\widetilde{W}_{n},W_{n}) &=& \bigg(\sup_{a \in \Gamma_{\widetilde{W}_{n}}} \inf_{b \in \Gamma_{W_{n}}} d(a,b) \bigg) \vee \bigg(\sup_{a \in \Gamma_{W_{n}}} \inf_{b \in \Gamma_{\widetilde{W}_{n}}} d(a,b) \bigg) \\[0.4em]
   &= :& U_{n} \vee T_{n},
\end{eqnarray*}
and therefore
\be\label{eq:AB}
\Pr [d_{M_{2}}(\widetilde{W}_{n}, W_{n})> \delta ] \leq \Pr(U_{n}>\delta) + \Pr(T_{n}>\delta).
\ee
To estimate the first term on the right hand side of (\ref{eq:AB}) we
use the same arguments as in the proof of Theorem 2.1 in Krizmani\'{c}~\cite{Kr19}, and thus we obtain
\begin{eqnarray}\label{eq:Un}
  \nonumber\{U_{n} > \delta\} & \subseteq & \{\exists\,a \in \Gamma_{\widetilde{W}_{n}} \ \textrm{such that} \ d(a,b) > \delta \ \textrm{for every} \ b \in \Gamma_{W_{n}} \} \\[0.4em]
  \nonumber & \subseteq & \{\exists\,k \in \{1,\ldots,q-1\} \ \textrm{such that} \ | \widetilde{W}_{n}(k/n) - W_{n}(k/n)| > \delta \}\\[0.4em]
  \nonumber & & \cup \ \{\exists\,k \in \{q,\ldots,n-q\} \ \textrm{such that} \ | \widetilde{W}_{n}(k/n) - W_{n}(k/n)| > \delta\\[0.4em]
  \nonumber & & \hspace*{1.5em} \textrm{and} \  | \widetilde{W}_{n}(k/n) - W_{n}((k+q)/n)| > \delta \}\\[0.4em]
  \nonumber & & \cup \ \{\exists\,k \in \{n-q+1,\ldots,n\} \ \textrm{such that} \ | \widetilde{W}_{n}(k/n) - W_{n}(k/n)| > \delta \}\\[0.4em]
  & =: & A^{U}_{n} \cup B^{U}_{n} \cup C^{U}_{n}.
\end{eqnarray}
For $k < q$ and an arbitrary $M>0$, by stationarity of the sequence $(Z_{i})$ we have
\begin{eqnarray*}
\Pr \Big( \Big| \widetilde{W}_{n} \Big( \frac{k}{n} \Big) \Big| > \frac{\delta}{2} \Big) & \leq & \Pr \Big( \sum_{i=1}^{q}\frac{C^{S}Z_{i}^{2}}{a_{n}^{2}} > \frac{\delta}{2} \Big)\\[0.2em]
 & = & \Pr \Big( \sum_{i=1}^{q} \frac{C^{S}Z_{i}^{2}}{a_{n}^{2}} > \frac{\delta}{2},\,C^{S} \leq M \Big) + \Pr \Big( \sum_{i=1}^{q}\frac{C^{S}Z_{i}^{2}}{a_{n}^{2}} > \frac{\delta}{2},\,C^{S} > M \Big)\\[0.2em]
 & \leq & \sum_{i=1}^{q} \Pr \Big( \frac{Z_{i}^{2}}{a_{n}^{2}} > \frac{\delta}{2qM} \Big) + \Pr(C^{S}>M)\\[0.2em]
 &=& q \Pr \Big( \frac{|Z_{1}|}{a_{n}} > \sqrt{\frac{\delta}{2qM}} \Big) + \Pr(C^{S}>M)
\end{eqnarray*}
and
\begin{eqnarray*}
\Pr \Big( \Big| W_{n} \Big( \frac{k}{n} \Big) \Big| > \frac{\delta}{2} \Big) &=& \Pr \Big(\sum_{i=1}^{k}\frac{X_{i}^{2}}{a_{n}^{2}} > \frac{\delta}{2} \Big) \leq \Pr \Big( \sum_{i=1}^{q} \frac{X_{i}^{2}}{a_{n}^{2}} > \frac{\delta}{2} \Big)\\[0.2em]
&\leq & \sum_{i=1}^{q} \Pr \Big( \frac{X_{i}^{2}}{a_{n}^{2}} > \frac{\delta}{2q} \Big) \leq \sum_{i=1}^{q} \Pr \Big( \frac{|X_{i}|}{a_{n}} > \frac{\sqrt{\delta}}{\sqrt{2q}} \Big)\\[0.2em]
&\leq & \sum_{i=1}^{q} \Pr \Big( \sum_{j=0}^{q} \frac{|C_{j}Z_{i-j}|}{a_{n}} > \frac{\sqrt{\delta}}{\sqrt{2q}} \Big)\\[0.2em]
 &\leq & \sum_{i=1}^{q} \Pr \Big( \sum_{j=0}^{q} \frac{|C_{j}Z_{i-j}|}{a_{n}} > \frac{\sqrt{\delta}}{\sqrt{2q}},\,C^{S} \leq M \Big) + \Pr(C^{S} > M)\\[0.2em]
&\leq & \sum_{i=1}^{q} \sum_{j=0}^{q} \Pr \Big( \frac{|Z_{i-j}|}{a_{n}} > \frac{\sqrt{\delta}}{(q+1)\sqrt{2qM}} \Big) + \Pr(C^{S} > M)\\[0.2em]
&= & q(q+1) \Pr \Big( \frac{|Z_{1}|}{a_{n}} > \frac{\sqrt{\delta}}{(q+1)\sqrt{2qM}} \Big) + \Pr(C^{S} > M),
\end{eqnarray*}
where in the last inequality above we used the fact that on the event $\{C^{S} \leq M\}$ it holds that $|C_{j}| \leq \sqrt{M}$ for all $j$.
Therefore
\begin{eqnarray}\label{eq:Bnlemmafirst}
  \nonumber \Pr (A^{U}_{n}) & \leq &  \sum_{k=1}^{q-1} \Pr \Big( \Big| \widetilde{W}_{n} \Big( \frac{k}{n} \Big) - W_{n} \Big( \frac{k}{n} \Big) \Big| >\delta \Big)\\[0.3em]
  \nonumber  &\leq& \sum_{k=1}^{q} \Big[ \Pr \Big( \Big| \widetilde{W}_{n} \Big( \frac{k}{n} \Big) \Big| > \frac{\delta}{2} \Big) + \Pr \Big( \Big| W_{n} \Big( \frac{k}{n} \Big) \Big| > \frac{\delta}{2} \Big) \Big]\\[0.3em]
   \nonumber &\leq & q^{2} \Pr \Big( \frac{|Z_{1}|}{a_{n}} > \sqrt{\frac{\delta}{2qM}} \Big) + q^{2}(q+1) \Pr \Big( \frac{|Z_{1}|}{a_{n}} > \frac{\sqrt{\delta}}{(q+1)\sqrt{2qM}} \Big)\\[0.6em]
   && +\,2q \Pr(C^{S}>M).
\end{eqnarray}
By the
 regular variation property
it holds that
\begin{equation*}
\lim_{n \to \infty} \Pr \Big( \frac{|Z_{1}|}{a_{n}} > \sqrt{\frac{\delta}{2qM}} \Big) =0 \qquad \textrm{and} \qquad \lim_{n \to \infty} \Pr \Big( \frac{|Z_{1}|}{a_{n}} > \frac{\sqrt{\delta}}{(q+1)\sqrt{2qM}} \Big)=0,
\end{equation*}
and hence from (\ref{eq:Bnlemmafirst}) we get
$$ \limsup_{n \to \infty} \Pr (A^{U}_{n}) \leq 2q \Pr(C^{S}>M).$$
Letting $M \to \infty$ we conclude
 \be\label{eq:setBn1}
 \lim_{n \to \infty} \Pr (A^{U}_{n}) = 0.
 \ee
Similarly
\be\label{eq:setBn4}
   \lim_{n \to \infty} \Pr (C^{U}_{n})=0.
\ee
By Lemma 2.3 (ii) in Basrak and Krizmani\'{c}~\cite{BaKr}, for $k \geq q$ the difference
$$ \sum_{i=1}^{k} \Big( \sum_{j=0}^{q} C_{j}^{2} \Big) \frac{Z_{i}^{2}}{a_{n}^{2}} - \sum_{i=1}^{k} \sum_{j=0}^{q} \frac{C_{j}^{2}Z_{i-j}^{2}}{a_{n}^{2}}$$
can be represented as
$$ \sum_{u=0}^{q-1}\frac{Z_{k-u}^{2}}{a_{n}^{2}} \sum_{s=u+1}^{q}C_{s}^{2} - \sum_{u=0}^{q-1}\frac{Z_{-u}^{2}}{a_{n}^{2}} \sum_{s=u+1}^{q}C_{s}^{2} =: H_{n}(k) - G_{n}.$$
Therefore
\begin{eqnarray}\label{e:Wn1}
\nonumber \widetilde{W}_{n} \Big( \frac{k}{n} \Big) -  W_{n} \Big( \frac{k}{n} \Big) &=& \frac{1}{a_{n}^{2}} \sum_{i=1}^{k} C^{S}Z_{i}^{2}- \frac{1}{a_{n}^{2}} \sum_{i=1}^{k} \Big( \sum_{j=0}^{q} C_{j}Z_{i-j} \Big)^{2}\\[0.4em]
 \nonumber &\hspace*{-12em} =& \hspace*{-6em} \sum_{i=1}^{k} \Big( \sum_{j=0}^{q} C_{j}^{2} \Big) \frac{Z_{i}^{2}}{a_{n}^{2}} - \sum_{i=1}^{k} \sum_{j=0}^{q} \frac{C_{j}^{2}Z_{i-j}^{2}}{a_{n}^{2}} -  \sum_{i=1}^{k} \sum_{0\leq j <s \leq q} \frac{2C_{j}C_{s}Z_{i-j}Z_{i-s}}{a_{n}^{2}}\\[0.3em]
 &\hspace*{-12em} =& \hspace*{-6em} H_{n}(k) - G_{n} -  \sum_{i=1}^{k} \sum_{0\leq j <s \leq q} \frac{2C_{j}C_{s}Z_{i-j}Z_{i-s}}{a_{n}^{2}}.
\end{eqnarray}
Similarly, for $q \leq k \leq n-q$ Lemma 2.3 (ii) in~\cite{BaKr} yields
\begin{eqnarray*}
\sum_{i=1}^{k} \Big( \sum_{j=0}^{q} C_{j}^{2} \Big) \frac{Z_{i}^{2}}{a_{n}^{2}} - \sum_{i=1}^{k+q} \sum_{j=0}^{q} \frac{C_{j}^{2}Z_{i-j}^{2}}{a_{n}^{2}} &=&  - \sum_{u=0}^{q-1}\frac{Z_{-u}^{2}}{a_{n}^{2}} \sum_{s=u+1}^{q}C_{s}^{2} - \sum_{u=1}^{q}\frac{Z_{k+u}^{2}}{a_{n}2^{}} \sum_{s=0}^{q-u}C_{s}^{2}\\[0.5em]
&=& =: -G_{n} - T_{n}(k),
\end{eqnarray*}
and hence
\begin{eqnarray}\label{e:Wn2}
\nonumber \widetilde{W}_{n} \Big( \frac{k}{n} \Big) -  W_{n} \Big( \frac{k+q}{n} \Big) &=& \frac{1}{a_{n}^{2}} \sum_{i=1}^{k} C^{S}Z_{i}^{2}- \frac{1}{a_{n}^{2}} \sum_{i=1}^{k+q} \Big( \sum_{j=0}^{q} C_{j}Z_{i-j} \Big)^{2}\\[0.4em]
 \nonumber &\hspace*{-12em} =& \hspace*{-6em} \sum_{i=1}^{k} \Big( \sum_{j=0}^{q} C_{j}^{2} \Big) \frac{Z_{i}^{2}}{a_{n}^{2}} - \sum_{i=1}^{k+q} \sum_{j=0}^{q} \frac{C_{j}^{2}Z_{i-j}^{2}}{a_{n}^{2}} - \sum_{i=1}^{k+q} \sum_{0\leq j <s \leq q} \frac{2C_{j}C_{s}Z_{i-j}Z_{i-s}}{a_{n}^{2}}\\[0.3em]
 &\hspace*{-12em} =& \hspace*{-6em} - G_{n} - T_{n}(k) - \sum_{i=1}^{k+q} \sum_{0\leq j <s \leq q} \frac{2C_{j}C_{s}Z_{i-j}Z_{i-s}}{a_{n}^{2}}.
\end{eqnarray}
Let $I_{n}(k) = \sum_{i=1}^{k} \sum_{0\leq j <s \leq q} 2C_{j}C_{s}Z_{i-j}Z_{i-s}/a_{n}^{2}$. By (\ref{e:Wn1}) and (\ref{e:Wn2}) we obtain
\begin{eqnarray}\label{e:BnUnew1}
\nonumber \Pr(B_{n}^{U}) & \leq & \Pr \Big( \exists\,k \in \{q,\ldots,n-q\} \ \textrm{such that} \ |H_{n}(k)-G_{n} - I_{n}(k)| > \delta \\[0.4em]
  \nonumber & & \hspace*{2em} \textrm{and} \ |-G_{n}-T_{n}(k)-I_{n}(k+q)| > \delta \Big)\\[0.4em]
  \nonumber & \leq & \Pr \Big( |G_{n}|> \frac{\delta}{3} \Big) + \Pr \Big(\exists\,k \in \{q,\ldots,n-q\} \ \textrm{such that} \ |I_{n}(k)| > \frac{\delta}{3}\Big)\\[0.4em]
  \nonumber &&  +\Pr \Big(\exists\,k \in \{q,\ldots,n-q\} \ \textrm{such that} \ |I_{n}(k+q)| > \frac{\delta}{3} \Big)\\[0.4em]
  &&  + \sum_{k=q}^{n-q} \Pr \Big( |H_{n}(k)| >
       \frac{\delta}{3} \ \textrm{and} \ |T_{n}(k)| > \frac{\delta}{3} \Big).
\end{eqnarray}
Take an arbitrary $M>0$, and observe that
\begin{eqnarray*}\label{e:BnUnew2}
   \Pr \Big( |G_{n}|> \frac{\delta}{3},\,C^{S} \leq M \Big) & \leq & \Pr \Big( C^{S} \sum_{u=0}^{q-1}\frac{Z_{-u}^{2}}{a_{n}^{2}}> \frac{\delta}{3},\,C^{S} \leq M \Big) \leq \Pr \Big( \sum_{u=0}^{q-1}\frac{Z_{-u}^{2}}{a_{n}}> \frac{\delta}{3M} \Big) \\[0.4em]
   & \leq &  q \Pr \Big( \frac{Z_{1}^{2}}{a_{n}^{2}} > \frac{\delta}{3qM} \Big) = q \Pr \Big( \frac{|Z_{1}|}{a_{n}} > \sqrt{\frac{\delta}{3qM}} \Big),
\end{eqnarray*}
and similarly
\begin{eqnarray*}\label{e:HTnew}
  \nonumber \Pr \Big( |H_{n}(k)| >
       \frac{\delta}{3} \ \textrm{and} \ |T_{n}(k)| > \frac{\delta}{3},\,C^{S} \leq M \Big) & & \\[0.5em]
   \nonumber & \hspace*{-22em} \leq & \hspace*{-11em} \Pr \Big( \sum_{u=0}^{q-1}\frac{Z_{k-u}^{2}}{a_{n}^{2}} >
       \frac{\delta}{3M} \ \textrm{and} \ \sum_{u=1}^{q}\frac{Z_{k+u}^{2}}{a_{n}^{2}} > \frac{\delta}{3M}\Big)\\[0.5em]
   \nonumber & \hspace*{-22em} = & \hspace*{-11em} \Pr \Big( \sum_{u=0}^{q-1}\frac{Z_{k-u}^{2}}{a_{n}^{2}} >
       \frac{\delta}{3M} \Big) \Pr \Big(\sum_{u=1}^{q}\frac{Z_{k+u}^{2}}{a_{n}^{2}} > \frac{\delta}{3M}\Big)\\[0.5em]
   & \hspace*{-22em} \leq & \hspace*{-11em} \bigg[ q \Pr \Big(\frac{|Z_{1}|}{a_{n}} >
       \sqrt{\frac{\delta}{3qM}} \Big) \bigg]^{2},
\end{eqnarray*}
where the equality above holds since the random variables $Z_{i}$ are independent. Hence
\begin{equation*}
 \sum_{k=q}^{n-q} \Pr \Big( |H_{n}(k)| >
       \frac{\delta}{3} \ \textrm{and} \ |T_{n}(k)| > \frac{\delta}{3},\,C^{S} \leq M \Big) \leq \frac{q}{n}\bigg[ n \Pr \Big(\frac{|Z_{1}|}{a_{n}} >
       \sqrt{\frac{\delta}{3qM}} \Big) \bigg]^{2}.
\end{equation*}
An application of the regular variation property yields
\begin{equation}\label{e:BnUnew2}
   \lim_{n \to \infty }\Pr \Big( |G_{n}|> \frac{\delta}{3},\,C^{S} \leq M \Big) =0,
\end{equation}
and
\begin{equation}\label{e:BnUnew3}
 \lim_{n \to \infty} \sum_{k=q}^{n-q} \Pr \Big( |H_{n}(k)| >
       \frac{\delta}{3} \ \textrm{and} \ |T_{n}(k)| > \frac{\delta}{3},\,C^{S} \leq M \Big) = 0.
\end{equation}
Using Markov's inequality
we obtain
\begin{eqnarray*}
  \Pr \Big(\exists\,k \in \{q,\ldots,n-q\} \ \textrm{such that} \ |I_{n}(k)|> \frac{\delta}{3},\,C^{S} \leq M \Big) & &\\[0.4em]
   & \hspace*{-46em} \leq & \hspace*{-23em} \Pr \Big(\exists\,k \in \{q,\ldots,n-q\} \ \textrm{such that} \ \sum_{i=1}^{k} \sum_{0\leq j <s \leq q} \frac{2|C_{j}C_{s}Z_{i-j}Z_{i-s}|}{a_{n}^{2}} > \frac{\delta}{3},\,C^{S} \leq M \Big)\\[0.4em]
   & \hspace*{-46em} \leq & \hspace*{-23em} \Pr \Big( \sum_{i=1}^{n-q} \sum_{0\leq j <s \leq q} \frac{|Z_{i-j}Z_{i-s}|}{a_{n}^{2}} > \frac{\delta}{6M}\Big)\\[0.4em]
    &\hspace*{-46em} \leq & \hspace*{-23em}  \Big( \frac{\delta}{6M} \Big)^{-r} \mathrm{E} \bigg( \sum_{i=1}^{n-q} \sum_{0\leq j <s \leq q} \frac{|Z_{i-j}Z_{i-s}|}{a_{n}^{2}} \bigg)^{r},
\end{eqnarray*}
for some $r \in ( \alpha/2, 1 \wedge \alpha)$. Applying the triangle inequality $|\sum_{i=1}^{\infty}a_{i}|^{r} \leq \sum_{i=1}^{\infty}|a_{i}|^{r}$ for $r \in (0,1]$ and real numbers $a_{1}, a_{2}, \ldots$, and the fact that $(Z_{i})$ is an i.i.d.~sequence we get
\begin{eqnarray}\label{e:BnU}
 \nonumber \Pr \Big(\exists\,k \in \{q,\ldots,n-q\} \ \textrm{such that} \ |I_{n}(k)|> \frac{\delta}{3},\,C^{S} \leq M \Big) & &\\[0.4em]
  \nonumber &\hspace*{-26em} \leq & \hspace*{-13em}  \Big( \frac{\delta}{6M} \Big)^{-r}  \sum_{i=1}^{n-q} \sum_{0\leq j <s \leq q} \mathrm{E} \bigg(\frac{|Z_{i-j}Z_{i-s}|}{a_{n}^{2}} \bigg)^{r}\\[0.4em]
  \nonumber &\hspace*{-26em} = & \hspace*{-13em}  \Big( \frac{\delta}{6M} \Big)^{-r}  \sum_{i=1}^{n-q} \sum_{0\leq j <s \leq q} \frac{\mathrm{E}|Z_{i-j}|^{r} \mathrm{E}|Z_{i-s}|^{r}}{a_{n}^{2r}}\\[0.4em]
  &\hspace*{-26em} = & \hspace*{-13em}  \Big( \frac{\delta}{6M} \Big)^{-r}  (n-q) {q+1 \choose 2} \frac{[\mathrm{E}|Z_{1}|^{r}]^{2}}{a_{n}^{2r}}.
\end{eqnarray}
Since $r < \alpha$ it holds that $\mathrm{E}|Z_{1}|^{r} < \infty$.
By (\ref{e:regvar}) and (\ref{eq:niz}) we have
\begin{equation}\label{e:App1}
 \lim_{n \to \infty} n a_{n}^{-\alpha} L(a_{n}) =1.
\end{equation}
Since $L$ is a slowly varying function, it holds that for all $s>0$ and $t \in \mathbb{R}$, $x^{s}[L(x)]^{t} \to \infty$ and $x^{-s}[L(x)]^{t} \to 0$ as $x \to \infty$ (Bingham et al.~\cite{BiGoTe89}, Proposition 1.3.6). Hence $a_{n}^{2r-\alpha}L(a_{n}) \to \infty$ as $n \to \infty$, and since by (\ref{e:App1})
$$ \lim_{n \to \infty} \frac{n}{a_{n}^{2r}}\,a_{n}^{2r-\alpha}L(a_{n})=1,$$
it follows that $n/a_{n}^{2r} \to 0$ as $n \to \infty$. This and relation (\ref{e:BnU}) imply
\begin{equation}\label{e:BnUnew4}
\lim_{n \to \infty} \Pr \Big(\exists\,k \in \{q,\ldots,n-q\} \ \textrm{such that} \ |I_{n}(k)|> \frac{\delta}{3},\,C^{S} \leq M \Big) = 0,
\end{equation}
and similarly
\begin{equation}\label{e:BnUnew5}
\lim_{n \to \infty} \Pr \Big(\exists\,k \in \{q,\ldots,n-q\} \ \textrm{such that} \ |I_{n}(k+q)|> \frac{\delta}{3},\,C^{S} \leq M \Big) = 0.
\end{equation}
 Therefore by relations (\ref{e:BnUnew1}), (\ref{e:BnUnew2}), (\ref{e:BnUnew3}), (\ref{e:BnUnew4}) and (\ref{e:BnUnew5}) we obtain
 $$ \limsup_{n \to \infty} \Pr (B^{U}_{n}) \leq \limsup_{n \to \infty} \Pr (B^{U}_{n}\cap \{ C^{S} > M \}) \leq \Pr (C^{S} > M),$$
and letting again $M \to \infty$ we conclude
 \be\label{eq:setBn3}
 \lim_{n \to \infty} \Pr (B^{U}_{n}) = 0.
 \ee
From relations (\ref{eq:Un}), (\ref{eq:setBn1}), (\ref{eq:setBn4}) and (\ref{eq:setBn3}) we obtain
  \be\label{eq:Unend}
  \lim_{n \to \infty} \Pr(U_{n} > \delta ) =0.
  \ee

It remains to estimate the second term on the right hand side of (\ref{eq:AB}). By applying the arguments from the proof of Theorem 2.1 in Krizmani\'{c}~\cite{Kr19} we can write
\begin{eqnarray}\label{eq:Zn}
  \nonumber\{T_{n} > \delta \} & \subseteq & \{\exists\,a \in \Gamma_{W_{n}} \ \textrm{such that} \ d(a,b) > \delta \ \textrm{for every} \ b \in \Gamma_{\widetilde{W}_{n}} \} \\[0.6em]
  \nonumber & \subseteq & \{\exists\,k \in \{1,\ldots,2q-1\} \
\ \textrm{such that} \ | W_{n}(k/n) - \widetilde{W}_{n}(k/n)| > \delta \}\\[0.6em]
  \nonumber & & \cup \ \Big\{\exists\,k \in \{2q,\ldots,n\} \
   \textrm{such that} \
  \widetilde{d}(W_{n}(k/n), [\widetilde{W}_{k}^{\min}, \widetilde{W}_{k}^{\max}]) > \delta
    \Big\}\\[0.6em]
  & =: & A^{T}_{n} \cup B^{T}_{n},
\end{eqnarray}
where $\widetilde{d}$ is the Euclidean metric on $\mathbb{R}$,
 $$\widetilde{W}^{\min}_k := \min\Big\{ \widetilde{W}_n \Big( \frac{k-q}{n} \Big), \widetilde{W}_n \Big( \frac{k}{n} \Big) \Big\} = \widetilde{W}_n \Big( \frac{k-q}{n} \Big)$$
 and
$$\widetilde{W}^{\max}_k := \max\Big\{ \widetilde{W}_n \Big( \frac{k-q}{n} \Big), \widetilde{W}_n \Big( \frac{k}{n} \Big) \Big\} = \widetilde{W}_n \Big( \frac{k}{n} \Big).$$
Similarly as before for the set $A^{U}_{n}$ we can show
\be\label{eq:Cnfirst}
\lim_{n \to \infty} \Pr( A^{T}_{n})=0.
\ee
Note that $\Pr (B_n^T)$ is bounded above by
\begin{eqnarray*}
\lefteqn{ \Pr \left(
\exists\,k \in \{2q,\ldots,n\} \ \textrm{such that} \
W_{n} \Big( \frac{k}{n} \Big) > \widetilde{W}^{\max}_k + \delta \right)}\\
&+&
\Pr \left(
\exists\,k \in \{2q,\ldots,n\} \ \textrm{such that} \ W_{n} \Big( \frac{k}{n} \Big)
 < \widetilde{W}^{\min}_k - \delta \right)\,.
\end{eqnarray*}
In the sequel we consider only the first of these two probabilities,
since the other one can be handled in a similar manner.
By using the same arguments as before when dealing with the set $B_{n}^{U}$ (in relations (\ref{e:Wn1}), (\ref{e:Wn2}) and (\ref{e:BnUnew1})), the first probability above
can be bounded by
\begin{eqnarray*}\label{e:BnTnew1}
\nonumber && \hspace*{-5.5em} \Pr \Big(\exists\,k \in \{2q,\ldots,n\} \ \textrm{such that} \
G_n - H_n(k) + I_{n}(k) > \delta \ \textrm{and} \ G_n + T_n(k-q) + I_{n}(k) > \delta \Big)\\[0.4em]
  \nonumber & \hspace*{3em} \leq &  \Pr \Big( G_{n}> \frac{\delta}{3} \Big) + \Pr \Big(\exists\,k \in \{2q,\ldots,n\} \ \textrm{such that} \ I_{n}(k) > \frac{\delta}{3}\Big)\\[0.4em]
 &  & + \sum_{k=2q}^{n} \Pr \Big( H_{n}(k) <
       - \frac{\delta}{3} \ \textrm{and} \ T_{n}(k-q) > \frac{\delta}{3} \Big).
\end{eqnarray*}
From the calculations yielding (\ref{e:BnUnew2}) and (\ref{e:BnUnew4}) we conclude that, as $n \to \infty$,
$$\Pr \Big( G_{n}> \frac{\delta}{3},\,C^{S} \leq M \Big) \to 0 \ \ \textrm{and} \ \Pr \Big(\exists\,k \in \{2q,\ldots,n\} \ \textrm{s.t.} \ I_{n}(k) > \frac{\delta}{3},\,C^{S} \leq M \Big) \to 0$$
for arbitrary $M>0$. Note that the last term on the right-hand side of (\ref{e:BnTnew1}) is equal to zero, since
 $$ H_n(k)= \sum_{u=0}^{q-1}\frac{Z_{k-u}^{2}}{a_{n}^{2}}\sum_{s=u+1}^{q}C_{s}^{2} \geq 0.$$
 Therefore
 $$ \limsup_{n \to \infty} \Pr (B^{T}_{n}) \leq \limsup_{n \to \infty} \Pr (B^{T}_{n}\cap \{ C^{S} > M \}) \leq \Pr (C^{S} > M),$$
 and by letting $M \to \infty$ we obtain
\begin{equation*}\label{eq:BnTnew2}
\lim_{n \to \infty} \Pr( B^{T}_{n})=0.
\end{equation*}
Together with relations (\ref{eq:Zn}) and (\ref{eq:Cnfirst}) this implies
\be\label{eq:Tnend}
\lim_{n \to \infty} \Pr(T_{n}>\delta)=0.
\ee
Now relations (\ref{eq:AB}), (\ref{eq:Unend}) and (\ref{eq:Tnend}) yield (\ref{e:SlutskyM}), and therefore $L_{n} \dto (C^{(1)}V, C^{(2)}W)$ as $n \to \infty$, in $D^{2}$ with the weak $M_{2}$ topology.
This concludes the proof.
\end{proof}

\begin{rem}\label{r:jointdepend1}
From the proof of Proposition~\ref{p:FLT} it follows that the components of the limiting process $(C^{(1)}V, C^{(2)}W)$ can be expressed as functionals of the limiting point process $N = \sum_{i}  \delta_{(T_{i}, P_{i}Q_{i})}$ from relation (\ref{e:BaTa1}), that is
$$ V(t) =  \lim_{u \to 0} \bigg( \sum_{T_{i} \le t} P_{i}Q_{i} 1_{\{ P_{i} > u \}} - t  \int_{u < |x| \leq 1} x\mu(\rmd x) \bigg) + t(p-r)\frac{\alpha}{1-\alpha} 1_{\{\alpha \neq 1\}}$$
and
$$ W(t) =  \lim_{u \to 0} \bigg( \sum_{T_{i} \le t} P_{i}^{2} 1_{\{ P_{i} > u \}} - t  \int_{u < |x| \leq 1} x^{2}\mu(\rmd x) \bigg) + t \frac{\alpha}{2-\alpha},$$
where the limits hold almost surely uniformly on $[0,1]$.
Since $\int_{ |x| \leq 1} x^{2}\mu(\rmd x) = \alpha/(2-\alpha)$ and according to the proofs of Theorem 2 in Davis~\cite{Da83} and Theorem 3.1 in Davis and Hsing~\cite{DaHs95} the points $P_{i}^{2}$, $i=1,2,\ldots$, are almost surely summable, it holds that
$ W(t) = \sum_{T_{i} \le t} P_{i}^{2}.$
Similarly
$ V(t) =  \sum_{T_{i} \le t} P_{i}Q_{i},$
but only when $\alpha <1$.
\end{rem}

\begin{rem}\label{r:M2M1}
Theorem~\ref{t:FinMA} establishes functional convergence of the joint stochastic process $L_{n}$ in the space $D^{2}$ endowed with the weak $M_{2}$ topology induced by the metric $d_{p}^{M_{2}}$ given in (\ref{e:defdp}). Since the second coordinate of $L_{n}$ is nondecreasing, the same arguments used in Remark 3 in Krizmani\'{c}~\cite{Kr18} yield the joint convergence of $L_{n}$ in the $M_{2}$ topology on the first coordinate and in the $M_{1}$ topology on the second coordinate.
\end{rem}

\section{Infinite order linear processes}
\label{S:InfiniteMA}

A standard idea when dealing with infinite order linear processes is to approximate them by a sequence of finite order linear processes, for which the weak convergence holds, and to show that the error of approximation is negligible in the limit. In our case, we will approximate them by finite order linear processes for which Theorem~\ref{t:FinMA} holds, and then we will show that the error of approximation is negligible with respect to the uniform metric. For the following result we will need several moment conditions similar to condition $(\ref{e:momcond})$. Note that condition (\ref{e:momcond}) is implied by the following moment condition
\begin{equation}\label{e:momcond2}
 \sum_{i=0}^{\infty}\mathrm{E}(|C_{i}|^{\delta} + |C_{i}|^{2 \delta}) < \infty \qquad \textrm{for some} \ 0 < \delta < \frac{\alpha}{2}.
\end{equation}


\begin{thm}\label{t:InfMA}
Let $(X_{i})$ be a linear process defined by
$$ X_{i} = \sum_{j=0}^{\infty}C_{j}Z_{i-j}, \qquad i \in \mathbb{Z},$$
where $(Z_{i})_{i \in \mathbb{Z}}$ is an i.i.d.~sequence of regularly varying random variables satisfying $(\ref{e:regvar})$ and $(\ref{eq:pq})$ with $\alpha \in (0,2)$. Assume conditions $(\ref{e:oceknula})$ and $(\ref{e:sim})$ hold. Let $(C_{i})_{i \geq 0}$ be a sequence of random variables, independent of $(Z_{i})$, satisfying conditions $(\ref{eq:InfiniteMAcond})$, $(\ref{e:momcond2})$,
\begin{equation}\label{c:momcs}
\sum_{j=0}^{\infty} (\mathrm{E}|C_{j}|^{2r})^{1/2} < \infty \qquad \textrm{for some}  \ r \in (\alpha/2, 1 \wedge \alpha),
\end{equation}
and
\begin{equation}\label{e:modnew1}
         \sum_{i=0}^{\infty}\mathrm{E}(|C_{i}|^{\gamma} + |C_{i}|^{2 \gamma}) < \infty
\end{equation}
for some $\gamma \in (\alpha, 1)$ if $\alpha \in (0,1)$ and $\gamma \in (\alpha/2, 1)$ if $\alpha \in [1,2)$.
If $\alpha \in [1,2)$ suppose further
\begin{equation}\label{eq:infmaTK3}
\sum_{j=0}^{\infty} \mathrm{E}|C_{j}| < \infty.
\end{equation}
Then, as $n \to \infty$,
\begin{equation}\label{e:mainconvInfMA}
L_{n}(\,\cdot\,) \dto  (C^{(1)}V(\,\cdot\,), C^{(2)}W(\,\cdot\,))
\end{equation}
 in $D^{2}$ with the weak $M_{2}$ topology,
 where $V$ is an $\alpha$--stable L\'{e}vy process with characteristic triple $(0, \mu, b)$, with $\mu$ as in $(\ref{eq:mu})$,
$$ b = \left\{ \begin{array}{cc}
                                   0, & \quad \alpha = 1,\\[0.4em]
                                   (p-r)\frac{\alpha}{1-\alpha}, & \quad \alpha \in (0,1) \cup (1,2),
                                 \end{array}\right.$$
with $p$ and $r$ defined in $(\ref{eq:pq})$, $W$ is an $\alpha/2$--stable L\'{e}vy process with characteristic triple $(0, \widetilde{\mu}, \alpha/(2-\alpha))$ with $\widetilde{\mu}(\rmd x) = (\alpha /2) x^{-\alpha/2-1} 1_{(0,\infty)}(x)\,\rmd x$,
and $(C^{(1)}, C^{(2)})$ is a random vector, independent of $(V, W)$, such that
$(C^{(1)}, C^{(2)}) \eind (C, C^{S})$, where $C=\sum_{i=0}^{\infty}C_{i}$ and $C^{S}=\sum_{i=0}^{\infty}C_{i}^{2}$.
\end{thm}

\begin{proof}
For $q \in \mathbb{N}$ define
$$ X_{i,q} = \sum_{j=0}^{q-1}C_{j}Z_{i-j} + C_{*}^{q} Z_{i-q}, \qquad i \in \mathbb{Z},$$
where $C_{*}^{q}= \sum_{j=q}^{\infty}C_{j}$, and let
$$ V_{n, q}(t) = \sum_{i=1}^{\floor{nt}} \frac{X_{i,q}}{a_{n}} \qquad \textrm{and} \qquad  W_{n, q}(t) = \sum_{i=1}^{\floor{nt}} \frac{X_{i,q}^{2}}{a_{n}^{2}}, \qquad t \in [0,1].$$
Condition (\ref{eq:InfiniteMAcond}) implies that coefficients $C_{0}, \ldots, C_{q-1}, C_{*}^{q}$ satisfy condition (\ref{eq:FiniteMAcond}), and hence an application of Theorem~\ref{t:FinMA} to the finite order linear process $(X_{i,q})_{i}$ yields that
\begin{equation}\label{e:SlutskyINFMA1}
 L_{n, q}(\,\cdot\,) := (V_{n, q}(\,\cdot\,), W_{n, q}(\,\cdot\,))  \dto L^{q}(\,\cdot\,) \qquad \textrm{as} \ n \to \infty,
 \end{equation}
in $D^{2}$ with the weak $M_{2}$ topology, with
$L^{q} =(C^{(1)}V, C^{(2)}_{q}W)$, where $V$ and $W$ are stable L\'{e}vy processes as in Theorem~\ref{t:FinMA}, and $(C^{(1)}, C^{(2)}_{q})$ is a random vector, independent of $(V, W)$, such that
$(C^{(1)}, C^{(2)}_{q}) \eind (C, C^{S,q})$, with
$$C^{S,q} = \sum_{j=0}^{q-1}C_{j}^{2} + (C_{*}^{q})^{2}.$$
Since $\sum_{j=0}^{\infty}|C_{j}| < \infty$ a.s.~(which follows from condition (\ref{e:momcond2})), it is straightforward to obtain
$ C^{S,q} \to C^{S}$
almost surely as $q \to \infty$. Therefore
$$ \| (B, B^{S,q}) - (B, B^{S})\|_{[0,1]} \to 0$$
almost surely as $q \to \infty$, where $B(t)= C$, $B^{S,q}(t)=C^{S,q}$ and $B^{S}(t)=C^{S}$ for $t \in [0,1]$. Since uniform convergence implies Skorokhod $M_{1}$ convergence, it follows that
$d_{p}^{M_{1}}( (B, B^{S,q}), (B, B^{S})  ) \to 0$
almost surely as $q \to \infty$, and hence, since almost sure convergence implies convergence in distribution, we have
\begin{equation*}\label{e:contmtKall}
(B, B^{S,q}) \dto (B, B^{S}) \qquad \textrm{as} \ q \to \infty,
\end{equation*}
in $D^{2}$ with the weak $M_{1}$ topology. An application of Theorem 3.29 and Corollary 5.18 in Kallenberg~\cite{Ka97} yields
\begin{equation}\label{e:contmtKall2}
 (B^{(1)}, B^{(2)}_{q}, V, W) \dto (B^{(1)}, B^{(2)}, V, W)
\end{equation}
in $D^{4}$ with the product $M_{1}$ topology, where $B^{(1)}(t)=C^{(1)}$, $B^{(2)}_{q}(t)=C^{(2)}_{q}$ and $B^{(2)}(t)=C^{(2)}$ for $t \in [0,1]$, and $(C^{(1)}, C^{(2)})$ is a random vector, independent of $(V, W)$, such that $(C^{(1)}, C^{(2)}) \eqd (C, C^{S})$.

Now, as in the proof of Proposition~\ref{p:FLT}, an application of the continuous mapping theorem to the convergence relation in (\ref{e:contmtKall2}), with the function $g$ given by
$$ g(x) = (x_{1}  x_{3}, x_{2}  x_{4}), \qquad x=(x_{1}, \ldots, x_{4}) \in D^{4},$$
yields
$$g (B^{(1)}, B^{(2)}_{q}, V, W) \dto g(B^{(1)}, B^{(2)}, V, W),$$
 as $q \to \infty$, that is
 \begin{equation}\label{e:SlutskyINFMA2}
L^{q} = (C^{(1)} V, C^{(2)}_{q} W) \dto L := (C^{(1)}V, C^{(2)}W)
 \end{equation}
in $D^{2}$ with the weak $M_{1}$ topology, and then also with the weak $M_{2}$ topology.

If we show that for every $\epsilon >0$
\begin{equation}\label{e:Slutskyinf01}
 \lim_{q \to \infty} \limsup_{n \to \infty}\Pr[d_{p}^{M_{2}}(L_{n}, L_{n,q})> \epsilon]=0,
\end{equation}
then from relations (\ref{e:SlutskyINFMA1}) and (\ref{e:SlutskyINFMA2}) by a generalization of Slutsky's theorem it will follow that $L_{n}(\,\cdot\,) \dto L(\,\cdot\,)$ in $D^{2}$ with the weak $M_{2}$ topology.  By the definition of the metric $d_{p}^{M_{2}}$ it suffices to show that
\begin{equation}\label{e:SlutskyINFV}
\lim_{q \to \infty} \limsup_{n \to \infty} \Pr [ d_{M_{2}}(V_{n}, V_{n,q}) > \epsilon ]=0
\end{equation}
and
\begin{equation}\label{e:SlutskyINFW}
\lim_{q \to \infty} \limsup_{n \to \infty} \Pr [ d_{M_{2}}(W_{n}, W_{n,q}) > \epsilon]=0.
\end{equation}
Since the metric $d_{M_{2}}$ is bounded above by the uniform metric note that
$$ \Pr [ d_{M_{2}}(V_{n}, V_{n,q}) > \epsilon]  \leq  \Pr \bigg( \sup_{t \in [0,1]}|V_{n}(t) - V_{n, q}(t)|> \epsilon \bigg)
  \leq  \Pr \bigg( \sum_{i=1}^{n}\frac{|X_{i}-X_{i,q}|}{a_{n}} > \epsilon \bigg).$$
In the case $\alpha \in (0,1)$ by repeating the arguments from the proof of Theorem 3.1 in Krizmani\'{c}~\cite{Kr19}, from conditions (\ref{e:momcond2}) and (\ref{e:modnew1}) we obtain
\begin{equation*}\label{e:infSlutsky}
 \lim_{q \to \infty} \limsup_{n \to \infty} \Pr \bigg( \sum_{i=1}^{n}\frac{|X_{i}-X_{i,q}|}{a_{n}} > \epsilon \bigg)=0,
\end{equation*}
and relation (\ref{e:SlutskyINFV}) holds.
In the case $\alpha \in [1,2)$ define
$$Z_{n,j}^{\leq} = \frac{Z_{j}}{a_{n}} 1_{\big\{ \frac{|Z_{j}|}{a_{n}} \leq 1 \big\}} - \mathrm{E} \Big( \frac{Z_{1}}{a_{n}} 1_{\big\{ \frac{|Z_{1}|}{a_{n}} \leq 1 \big\}} \Big) \quad \textrm{and} \quad Z_{n,j}^{>} =  \frac{Z_{j}}{a_{n}} 1_{\big\{ \frac{|Z_{j}|}{a_{n}} > 1 \big\}} + \mathrm{E} \Big( \frac{Z_{1}}{a_{n}} 1_{\big\{ \frac{|Z_{1}|}{a_{n}} \leq 1 \big\}} \Big) $$
  for $j \in \mathbb{Z}$ and $n \in \mathbb{N}$, and observe that $Z_{n,j}^{\leq} + Z_{n,j}^{>} = Z_{j}/a_{n}$. Let
$$ \widetilde{C}_{j} = \left\{ \begin{array}{cl}
                                   C_{j}, & \quad \textrm{if} \ j \geq q+1,\\[0.4em]
                                   C_{q}-C_{*}^{q}, & \quad \textrm{if} \ j=q.
                                 \end{array}\right.$$
and observe that
\begin{eqnarray*}
\nonumber  V_{n}(t) - V_{n,q}(t) & = & \sum_{i=1}^{\floor{nt}} \frac{1}{a_{n}} \bigg( \sum_{j=q}^{\infty}C_{j}Z_{i-j} -C_{*}^{q}Z_{i-q} \bigg) \\[0.4em]
& = &  \sum_{i=1}^{\floor{nt}}  \sum_{j=q}^{\infty} \widetilde{C}_{j}Z_{n,i-j}^{\leq} + \sum_{i=1}^{\floor{nt}}  \sum_{j=q}^{\infty} \widetilde{C}_{j}Z_{n,i-j}^{>}.
\end{eqnarray*}
Therefore
\begin{eqnarray}\label{eq:I1I2}
\nonumber   \Pr[d_{M_{2}}(V_{n}, V_{n,q})> \epsilon] & \leq & \Pr \bigg( \sup_{ t \in [0,1]} |V_{n}(t) - V_{n,q}(t)| > \epsilon \bigg)\\[0.4em]
\nonumber & \hspace*{-18em} \leq & \hspace*{-9em} \Pr \bigg( \max_{1 \leq l \leq n} \bigg| \sum_{i=1}^{l} \sum_{j=q}^{\infty} \widetilde{C}_{j}Z_{n,i-j}^{\leq} \bigg| > \frac{\epsilon}{2} \bigg) + \Pr \bigg( \max_{1 \leq l \leq n} \bigg| \sum_{i=1}^{l} \sum_{j=q}^{\infty} \widetilde{C}_{j}Z_{n,i-j}^{>} \bigg| > \frac{\epsilon}{2} \bigg)\\[0.4em]
& \hspace*{-18em} =: & \hspace*{-9em} I_{1} + I_{2}.
\end{eqnarray}
Now following the arguments from the proof of Theorem 3.1 in Krizmani\'{c}~\cite{Kr22-1} we obtain, by
H\"{o}lder's and Markov's inequalities and the fact that the sequence $(C_{i})_{i \geq 0}$ is independent of $(Z_{i})$,
\begin{equation}\label{e:TKcond}
I_{1} \leq \mathrm{E} \bigg( \sum_{j=q}^{\infty}|\widetilde{C}_{j}| \bigg) + \frac{4}{\epsilon^{2}} \sum_{j=q}^{\infty} \mathrm{E} |\widetilde{C}_{j}| \cdot \sup_{k \geq q} \mathrm{E} \bigg( \max_{1 \leq l \leq n} \bigg| \sum_{i=1}^{l}Z_{n,i-k}^{\leq} \bigg|^{2} \bigg).
\end{equation}
Since $(Z_{i})$ is a sequence of i.i.d.~random variables, by Corollary 2 in Tyran-Kami\'{n}ska~\cite{Ty10b} 
it holds that
$$\limsup_{n \to \infty} \sup_{k \geq 0} \mathrm{E} \bigg( \max_{1 \leq l \leq n} \bigg| \sum_{i=1}^{l}Z_{n,i-k}^{\leq} \bigg|^{2} \bigg) < \infty.$$
Therefore from (\ref{e:TKcond}), by noting that $ \sum_{j=q}^{\infty}|\widetilde{C}_{j}| \leq 2 \sum_{j=q}^{\infty}|C_{j}|$, we conclude that there exists a positive constant $D_{1}$ such that for all $q \in \mathbb{N}$ it holds that
\begin{equation}\label{eq:I1}
 \limsup_{n \to \infty} I_{1} \leq D_{1} \sum_{j=q}^{\infty}\mathrm{E} |C_{j}|.
\end{equation}
In order to estimate $I_{2}$ we consider separately the cases $\alpha \in (1,2)$ and $\alpha=1$. Assume first $\alpha \in (1,2)$. Note that by (\ref{e:oceknula})
$$ Z_{n,j}^{>} =  \frac{Z_{j}}{a_{n}} 1_{\big\{ \frac{|Z_{j}|}{a_{n}} > 1 \big\}} - \mathrm{E} \Big( \frac{Z_{1}}{a_{n}} 1_{\big\{ \frac{|Z_{1}|}{a_{n}} > 1 \big\}} \Big). $$
 Applying Markov's inequality, the fact that the sequence $(C_{i})_{i \geq 0}$ is independent of $(Z_{i})$ and the stationarity of the sequence $(Z_{i})$ we obtain
\begin{eqnarray}\label{eq:alpha1}
\nonumber I_{2} & \leq &   \Pr \bigg(  \sum_{i=1}^{n} \bigg| \sum_{j=q}^{\infty} \widetilde{C}_{j}Z_{n,i-j}^{>} \bigg| > \frac{\epsilon}{2} \bigg) \leq  \frac{2}{\epsilon}\,\mathrm{E} \bigg(  \sum_{i=1}^{n} \bigg| \sum_{j=q}^{\infty} \widetilde{C}_{j}Z_{n,i-j}^{>} \bigg|  \bigg)\\[0.4em]
 & \leq & \frac{4n}{\epsilon a_{n}}  \sum_{j=q}^{\infty} \mathrm{E} |\widetilde{C}_{j}| \cdot \mathrm{E} \Big( |Z_{1}| 1 _{\{ |Z_{1}|>a_{n} \}} \Big)
\end{eqnarray}
By Karamata's theorem
$$ \lim_{n \to \infty} \frac{n}{ a_{n}} \mathrm{E} \Big( |Z_{1}| 1 _{\{ |Z_{1}|>a_{n} \}} \Big) = \frac{\alpha}{\alpha-1},$$
 and hence from (\ref{eq:alpha1}) it follows that there exists a positive constant $D_{2}$ such that
\begin{equation}\label{eq:I2a}
 \limsup_{n \to \infty} I_{2} \leq D_{2} \sum_{j=q}^{\infty} \mathrm{E} |C_{j}|.
\end{equation}
Now assume $\alpha=1$. Markov's inequality implies
$$ I_{2} \leq  \frac{2^{\delta}}{\epsilon^{\delta}}  \mathrm{E} \bigg(  \sum_{i=1}^{n} \bigg| \sum_{j=q}^{\infty} \widetilde{C}_{j}Z_{n,i-j}^{>} \bigg| \bigg)^{\delta},$$
with $\delta$ as in relation (\ref{e:momcond2}). Note that by (\ref{e:sim}) it holds that $Z_{n,i-j}^{>} =  a_{n}^{-1}Z_{i-j}  1_{\{ |Z_{i-j}|>a_{n} \}}$.
Since $\delta < 1$, a double application of the triangle inequality $|\sum_{i=1}^{\infty}a_{i}|^{s} \leq \sum_{i=1}^{\infty}|a_{i}|^{s}$ with $s \in (0,1]$ yields
\begin{eqnarray*}
I_{2} & \leq & \frac{2^{\delta}}{\epsilon^{\delta}} \sum_{i=1}^{n} \mathrm{E} \bigg( \bigg| \sum_{j=q}^{\infty} \widetilde{C}_{j}Z_{n,i-j}^{>} \bigg|^{\delta} \bigg)
       \leq  \frac{2^{\delta}}{\epsilon^{\delta} a_{n}^{\delta}} \sum_{i=1}^{n} \sum_{j=q}^{\infty} \mathrm{E}  \bigg( \bigg| \widetilde{C}_{j}Z_{i-j} 1_{\{ |Z_{i-j}|>a_{n} \}} \bigg|^{\delta} \bigg).
\end{eqnarray*}
Using again the fact that $(C_{i})$ is independent of $(Z_{i})$ and the stationarity of $(Z_{i})$ we obtain
$$ I_{2} \leq \frac{2^{\delta} n}{\epsilon^{\delta} a_{n}^{\delta}} \mathrm{E} \Big( |Z_{1}|^{\delta} 1 _{\{ |Z_{1}|>a_{n} \}} \Big) \sum_{j=q}^{\infty} \mathrm{E} |\widetilde{C}_{j}|^{\delta}.$$
From this, since by Karamata's theorem
$$ \lim_{n \to \infty} \frac{n}{a_{n}^{\delta}} \mathrm{E} \Big( |Z_{1}|^{\delta} 1 _{\{ |Z_{1}|>a_{n} \}} \Big) = \frac{1}{1-\delta},$$
it follows that there exists a positive constant $D_{3}$ such that
\begin{equation*}\label{eq:I2b}
\limsup_{n \to \infty} I_{2} \leq D_{3} \sum_{j=q}^{\infty} \mathrm{E} |C_{j}|^{\delta}.
\end{equation*}
This together with (\ref{eq:I1I2}), (\ref{eq:I1}) and (\ref{eq:I2a}) shows that
$$ \limsup_{n \to \infty}\Pr[d_{M_{2}}(V_{n}, V_{n,q})> \epsilon] \leq D_{1} \sum_{j=q}^{\infty} \mathrm{E}|C_{j}| + (D_{2}+D_{3}) \sum_{j=q}^{\infty} \mathrm{E}|C_{j}|^{s},$$
 where
 $$ s = \left\{ \begin{array}{cc}
                                   \delta, & \quad \textrm{if} \ \alpha = 1,\\[0.4em]
                                   1, & \quad \textrm{if} \ \alpha \in (1,2).
                                 \end{array}\right.$$
Finally, the dominated convergence theorem and conditions (\ref{e:momcond2}) and (\ref{eq:infmaTK3}) imply relation (\ref{e:SlutskyINFV}) for $\alpha \in [1,2)$.

Now we turn our attention to relation (\ref{e:SlutskyINFW}).
Since
\begin{equation*}
 X_{i}^{2}  =  \Big( \sum_{j=0}^{\infty}C_{j}Z_{i-j} \Big)^{2} = \sum_{j=0}^{\infty}C_{j}^{2}Z_{i-j}^{2} + 2 \sum_{0 \leq j < s < \infty}C_{j}C_{s}Z_{i-j}Z_{i-s}
\end{equation*}
and
\begin{eqnarray*}
  X_{i,q}^{2} &=&  \Big( \sum_{j=0}^{q-1}C_{j}Z_{i-j} +  C_{*}^{q}Z_{i-q} \Big)^{2}\\[0.3em]
  &=& \sum_{j=0}^{q-1}C_{j}^{2}Z_{i-j}^{2} + (C_{*}^{q})^{2}Z_{i-q}^{2} + 2 \sum_{0 \leq j < s \leq q-1}C_{j}C_{s}Z_{i-j}Z_{i-s} +  2\sum_{j=0}^{q-1}C_{j}C_{*}^{q}Z_{i-j}Z_{i-q},
\end{eqnarray*}
we have
\begin{eqnarray*}
  \sum_{i=1}^{n}\frac{|X_{i}^{2}-X_{i,q}^{2}|}{a_{n}^{2}} &\leq &   \sum_{i=1}^{n} \sum_{j=q}^{\infty}\frac{C_{j}^{2}Z_{i-j}^{2}}{a_{n}^{2}} + 2 \sum_{i=1}^{n} \sum_{0 \leq j < s < \infty}\frac{|C_{j}C_{s}Z_{i-j}Z_{i-s}|}{a_{n}^{2}}\\[0.3em]
 & \hspace*{-13em} & \hspace*{-6.5em} +  \sum_{i=1}^{n} \frac{(C_{*}^{q})^{2}Z_{i-q}^{2}}{a_{n}^{2}} +  2\sum_{i=1}^{n} \sum_{0 \leq j < s \leq q-1}\frac{|C_{j}C_{s}Z_{i-j}Z_{i-s}|}{a_{n}^{2}} +  2 \sum_{i=1}^{n} \sum_{j=0}^{q-1}\frac{|C_{j}C_{*}^{q}Z_{i-j}Z_{i-q}|}{a_{n}^{2}}\\[0.6em]
  & \hspace*{-13em} =:& \hspace*{-6.5em} K_{1} + K_{2} + K_{3} + K_{4} + K_{5},
\end{eqnarray*}
and therefore
\begin{eqnarray}\label{e:M2Wnq}
\nonumber \Pr [ d_{M_{2}}(W_{n}, W_{n,q}) > \epsilon]  &\leq&  \Pr \bigg( \sup_{t \in [0,1]}|W_{n}(t) - W_{n, q}(t)|> \epsilon \bigg)\\[0.4em]
\nonumber & \hspace*{-14em} \leq & \hspace*{-7em} \Pr \bigg( \sup_{t \in [0,1]} \sum_{i=1}^{\lfloor nt \rfloor} \frac{|X_{i}^{2} - X_{i,q}^{2}|}{a_{n}^{2}}> \epsilon \bigg) = \Pr \bigg(\sum_{i=1}^{n} \frac{|X_{i}^{2} - X_{i,q}^{2}|}{a_{n}^{2}}> \epsilon \bigg)\\[0.4em]
& \hspace*{-14em} \leq & \hspace*{-7em} \sum_{i=1}^{5} \Pr \Big( K_{i} > \frac{\epsilon}{5} \Big).
\end{eqnarray}
In order to bound the probability $\Pr(K_{1} > \epsilon/5)$ note that
$$K_{1} = \sum_{i=-\infty}^{0}\frac{Z_{i-q+1}^{2}}{a_{n}^{2}} \sum_{j=0}^{n-1}C_{q-i+j}^{2} + \sum_{i=1}^{n-1}\frac{Z_{i-q+1}^{2} }{a_{n}^{2}} \sum_{j=i}^{n-1}C_{q-i+j}^{2}.$$
Let
$$  D_{i,n,q} =  \left\{ \begin{array}{cl}
                                   \displaystyle \sum_{j=0}^{n-1} C_{q-i+j}^{2}, & \quad  i \leq 0,\\[1.4em]
                                   \displaystyle  \sum_{j=i}^{n-i}C_{q-i+j}^{2}, & \quad i=1,\ldots,n-1,
                                 \end{array}\right.$$
and
$$ \widetilde{Z}^{\leq}_{i,n} = \frac{Z_{i}^{2}}{a_{n}^{2}} 1_{\big\{ \frac{|Z_{i}|}{a_{n}} \leq 1 \big\}} \qquad \textrm{and} \qquad
\widetilde{Z}^{>}_{i,n} = \frac{Z_{i}^{2}}{a_{n}^{2}} 1_{\big\{ \frac{|Z_{i}|}{a_{n}} > 1 \big\}}.$$
Then
\begin{equation}\label{e:Diqn1}
\Pr \Big( K_{1} > \frac{\epsilon}{5} \Big) \leq \Pr \bigg(\sum_{i=-\infty}^{n-1} D_{i,n,q} \widetilde{Z}^{\leq}_{i-q+1,n} > \frac{\epsilon}{10} \bigg) +
\Pr \bigg(\sum_{i=-\infty}^{n-1} D_{i,n,q} \widetilde{Z}^{>}_{i-q+1,n} > \frac{\epsilon}{10} \bigg).
\end{equation}
Using Markov's inequality, the triangle inequality $|\sum_{i=1}^{\infty}a_{i}|^{s} \leq \sum_{i=1}^{\infty}|a_{i}|^{s}$ with $s \in (0,1]$, the fact that $(C_{i})$ is independent of $(Z_{i})$ and the stationarity of the sequence $(Z_{i})$, for the first term on the right-hand side of (\ref{e:Diqn1}) we obtain
\begin{eqnarray*}
  \nonumber \Pr \bigg(\sum_{i=-\infty}^{n-1} D_{i,n,q} \widetilde{Z}^{\leq}_{i-q+1,n} > \frac{\epsilon}{10} \bigg) & \leq & \Big(\frac{\epsilon}{10} \Big)^{-\gamma} \mathrm{E} \bigg( \sum_{i=-\infty}^{n-1} |D_{i,n,q}\widetilde{Z}^{\leq}_{i-q+1,n}| \bigg)^{\gamma} \\[0.5em]
   & \leq & \Big(\frac{\epsilon}{10} \Big)^{-\gamma} \mathrm{E} \bigg( \sum_{i=-\infty}^{n-1} |D_{i,n,q}|^{\gamma}|\widetilde{Z}^{\leq}_{i-q+1,n}|^{\gamma} \bigg)\\[0.5em]
  & \leq & \Big(\frac{\epsilon}{10} \Big)^{-\gamma} \mathrm{E}|\widetilde{Z}^{\leq}_{1,n}|^{\gamma}  \sum_{i=-\infty}^{n-1} \mathrm{E}|D_{i,n,q}|^{\gamma}.
\end{eqnarray*}
Again by the triangle inequality we have
$$ \sum_{i=-\infty}^{n-1} \mathrm{E} |D_{i,n,q}|^{\gamma} \leq \sum_{i=-\infty}^{0}  \sum_{j=0}^{n-1} \mathrm{E}|C_{q-i+j}|^{2\gamma} + \sum_{i=1}^{n-1} \sum_{j=i}^{n-i} \mathrm{E} |C_{q-i+j}|^{2\gamma}.$$
Note that every $\mathrm{E}|C_{j}|^{2\gamma}$, for $j=q, q+1, \ldots$, appears in the sum $\sum_{i=-\infty}^{0}  \sum_{j=0}^{n-1} \mathrm{E}|C_{q-i+j}|^{2\gamma}$ at most $n$ times, and that the sum $\sum_{i=1}^{n-1} \sum_{j=i}^{n-i} \mathrm{E} |C_{q-i+j}|^{2\gamma}$ is bounded above by $(n-1) \sum_{j=q}^{\infty} \mathrm{E}|C_{j}|^{2\gamma}$. Hence
\begin{eqnarray}\label{e:Diqn2}
  \nonumber \Pr \bigg(\sum_{i=-\infty}^{n-1} D_{i,n,q} \widetilde{Z}^{\leq}_{i-q+1,n} > \frac{\epsilon}{10} \bigg)\\[0.4em]
   \nonumber & \hspace*{-18em} \leq & \hspace*{-9em} \Big(\frac{\epsilon}{10} \Big)^{-\gamma} \mathrm{E}|\widetilde{Z}^{\leq}_{1,n}|^{\gamma}  \bigg(
  n  \sum_{j=q}^{\infty} \mathrm{E} |C_{j}|^{2\gamma} + (n-1)  \sum_{j=q}^{\infty} \mathrm{E} |C_{j}|^{2\gamma} \bigg) \\[0.6em]
  & \hspace*{-18em} \leq & \hspace*{-9em} 2 \Big(\frac{\epsilon}{10} \Big)^{-\gamma} n \mathrm{E}|\widetilde{Z}^{\leq}_{1,n}|^{\gamma}  \sum_{j=q}^{\infty} \mathrm{E} |C_{j}|^{2\gamma}.
\end{eqnarray}
Similarly
\begin{equation}\label{e:Diqn3}
   \Pr \bigg(\sum_{i=-\infty}^{n-1} D_{i,n,q} \widetilde{Z}^{>}_{i-q+1,n} > \frac{\epsilon}{10} \bigg)
   \leq 2 \Big(\frac{\epsilon}{10} \Big)^{-\delta} n \mathrm{E}|\widetilde{Z}^{>}_{1,n}|^{\delta}  \sum_{j=q}^{\infty} \mathrm{E} |C_{j}|^{2\delta}.
\end{equation}
Since $2\gamma > \alpha$ and $2 \delta < \alpha$, By Karamata's theorem and (\ref{eq:niz}), as $n \to \infty$,
$$ n\mathrm{E}|\widetilde{Z}^{\leq}_{1,n}|^{\gamma} = \frac{\mathrm{E}(|Z_{1}|^{2\gamma}1_{\{ |Z_{1}| \leq a_{n}\}})}{a_{n}^{2\gamma} \Pr(|Z_{1}| > a_{n})} \cdot n \Pr(|Z_{1}| > a_{n}) \to \frac{\alpha}{2\gamma - \alpha} < \infty$$
and
$$ n\mathrm{E}|\widetilde{Z}^{>}_{1,n}|^{\delta} = \frac{\mathrm{E}(|Z_{1}|^{2\delta}1_{\{ |Z_{1}| > a_{n}\}})}{a_{n}^{2\delta} \Pr(|Z_{1}| > a_{n})} \cdot n \Pr(|Z_{1}| > a_{n}) \to \frac{\alpha}{\alpha - 2\delta} < \infty.$$
From this and relations (\ref{e:Diqn1}), (\ref{e:Diqn2}) and (\ref{e:Diqn3}) we conclude that
$$ \limsup_{n \to \infty} \Pr \Big( K_{1} > \frac{\epsilon}{5} \Big) \leq c_{1} \bigg( \sum_{j=q}^{\infty} \mathrm{E} |C_{j}|^{2\gamma} + \sum_{j=q}^{\infty} \mathrm{E} |C_{j}|^{2\delta} \bigg),$$
where $c_{1} =  2 (\epsilon/10)^{-\gamma} \alpha/(2\gamma - \alpha) +  2 (\epsilon/10)^{-\delta} \alpha/(\alpha - 2\delta) < \infty$. Now letting $q \to \infty$, conditions (\ref{e:momcond2}) and (\ref{e:modnew1}) imply
\begin{equation}\label{e:K1}
 \lim_{q \to \infty} \limsup_{n \to \infty} \Pr \Big(K_{1} > \frac{\epsilon}{5} \Big)=0.
\end{equation}
We handle $K_{3}$ similarly to obtain
$$ \limsup_{n \to \infty} \Pr \Big( K_{3} > \frac{\epsilon}{5} \Big) \leq c_{3} \bigg[ \bigg( \sum_{j=q}^{\infty} \mathrm{E} |C_{j}|^{\gamma} \bigg)^{2} + \bigg( \sum_{j=q}^{\infty} \mathrm{E} |C_{j}|^{\delta} \bigg)^{2} \bigg],$$
with $c_{3} =   (\epsilon/10)^{-\gamma} \alpha/(2\gamma - \alpha) +   (\epsilon/10)^{-\delta} \alpha/(\alpha - 2\delta)$, and again conditions (\ref{e:momcond2}) and (\ref{e:modnew1}) yield
\begin{equation}\label{e:K3}
 \lim_{q \to \infty} \limsup_{n \to \infty} \Pr \Big(K_{3} > \frac{\epsilon}{5} \Big)=0.
\end{equation}
Further, by Markov's inequality we obtain
\begin{equation*}
  \Pr \Big(K_{2} > \frac{\epsilon}{5} \Big)  \leq  \Big( \frac{\epsilon}{10} \Big)^{-r} \mathrm{E} \bigg( \sum_{i=1}^{n} \sum_{0 \leq j < s < \infty}\frac{|C_{j}C_{s}Z_{i-j}Z_{i-s}|}{a_{n}^{2}} \bigg)^{r}
\end{equation*}
with $r$ as in (\ref{c:momcs}). Applying the triangle inequality and the fact that $(Z_{i})$ is an i.i.d.~sequence independent of $(C_{i})$ we get
\begin{eqnarray*}
 \nonumber  \Pr \Big(K_{2} > \frac{\epsilon}{5} \Big) & \leq & \Big( \frac{\epsilon}{10} \Big)^{-r}  \sum_{i=1}^{n} \sum_{0 \leq j < s < \infty} \mathrm{E} \bigg( \frac{|C_{j}C_{s}Z_{i-j}Z_{i-s}|}{a_{n}^{2}} \bigg)^{r}\\[0.4em]
  \nonumber & = &  \Big( \frac{\epsilon}{10} \Big)^{-r}  \sum_{i=1}^{n} \sum_{0 \leq j < s < \infty} \mathrm{E}|C_{j}C_{s}|^{r} \frac{\mathrm{E}|Z_{i-j}|^{r} \mathrm{E} |Z_{i-s}|^{r}}{a_{n}^{2r}} \\[0.4em]
  \nonumber & = &   \Big( \frac{\epsilon}{10} \Big)^{-r} \frac{n}{a_{n}^{2r}} [\mathrm{E}|Z_{1}|^{r}]^{2}  \sum_{0 \leq j < s < \infty} \mathrm{E}|C_{j}C_{s}|^{r}.
\end{eqnarray*}
Note that $\mathrm{E}|Z_{1}|^{r} < \infty$ since $r < \alpha$. By H\"{o}lder's inequality and condition (\ref{c:momcs}) we have
\begin{eqnarray*}
 \sum_{0 \leq j < s < \infty} \mathrm{E}|C_{j}C_{s}|^{r} &\leq&  \sum_{0 \leq j < s < \infty} \Big( \mathrm{E}|C_{j}|^{2r} \mathrm{E}|C_{s}|^{2r} \Big)^{1/2} = \sum_{j=0}^{\infty} \sum_{s=j+1}^{\infty} \Big( \mathrm{E}|C_{j}|^{2r} \mathrm{E}|C_{s}|^{2r} \Big)^{1/2}\\[0.4em]
 &\leq& \bigg( \sum_{j=0}^{\infty} (\mathrm{E}|C_{j}|^{2r})^{1/2} \bigg)^{2} < \infty,
\end{eqnarray*}
As in the proof of Theorem~\ref{t:FinMA} (after relation (\ref{e:App1})) it holds that
$n/a_{n}^{2r} \to 0$ as $n \to \infty$. Therefore
\begin{equation}\label{e:K2}
\lim_{n \to \infty} \Pr \Big(K_{2} > \frac{\epsilon}{5} \Big)=0,
\end{equation}
and similarly
\begin{equation}\label{e:K4}
\lim_{n \to \infty} \Pr \Big(K_{4} > \frac{\epsilon}{5} \Big)=0.
\end{equation}
Since by the definition of $C_{*}^{q}$
$$ K_{5} = 2 \sum_{i=1}^{n} \sum_{j=0}^{q-1}\frac{|C_{j}C_{*}^{q}Z_{i-j}Z_{i-q}|}{a_{n}^{2}} \leq 2 \sum_{i=1}^{n} \sum_{j=0}^{q-1} \sum_{s=q}^{\infty} \frac{|C_{j}C_{s}Z_{i-j}Z_{i-q}|}{a_{n}^{2}},$$
we analogously obtain
\begin{equation}\label{e:K5}
\lim_{n \to \infty} \Pr \Big(K_{5} > \frac{\epsilon}{5} \Big)=0.
\end{equation}
Now from (\ref{e:M2Wnq}) and relations (\ref{e:K1})-(\ref{e:K5}) we obtain (\ref{e:SlutskyINFW}), and therefore we conclude that $L_{n} \dto L$ in $D^{2}$ with the weak $M_{2}$ topology.
\end{proof}

\begin{rem}
In the case when the sequence of coefficients $(C_{j})$ is deterministic, condition (\ref{e:momcond2}) reduces to
$\sum_{i=0}^{\infty}|C_{i}|^{\delta} < \infty$, since the latter implies $\sum_{i=0}^{\infty}|C_{i}|^{2\delta} < \infty$. In general $\sum_{i=0}^{\infty}|C_{i}|^{\lambda_{1}} < \infty$ implies $\sum_{i=0}^{\infty}|C_{i}|^{\lambda_{2}} < \infty$ for $0 < \lambda_{1} < \lambda_{2}$. Therefore in this case conditions (\ref{c:momcs}), (\ref{e:modnew1}) and (\ref{eq:infmaTK3}) in Theorem~\ref{t:InfMA} can be dropped since they are all implied by $\sum_{i=0}^{\infty}|C_{i}|^{\delta} < \infty$. This in general does not hold when the coefficients are random (see Krizmani\'{c}~\cite{Kr22-1}, Remark 3.1, and Krizmani\'{c}~\cite{Kr19}, p.~739).
\end{rem}

\section{Self-normalized partial sum processes}
\label{S:Selfnormalized}

As noted in Remark~\ref{r:M2M1}, under the conditions from Theorem~\ref{t:InfMA}, the joint convergence of $L_{n}=(V_{n}, W_{n})$, for infinite order linear processes, actually holds in the $M_{2}$ topology on the first coordinate and in the $M_{1}$ topology on the second coordinate. So if we consider only the second coordinate of $L_{n}$, we have $W_{n} \dto W$ in $D^{1}$ with the $M_{1}$ topology. The function $\pi \colon D^{1} \to \mathbb{R}$, defined by $\pi(x)=x(1)$, is continuous with respect to the $M_{1}$ topology on $D^{1}$ (see Theorem 12.5.1 in Whitt~\cite{Whitt02}), and therefore the continuous mapping theorem yields $\pi(W_{n}) \dto \pi(W)$ as $n \to \infty$, that is
$$  W_{n}(1) = \sum_{i = 1}^{n} \frac{X_{i}^{2}}{a_{n}^{2}} \dto C^{(2)}W(1).$$

\begin{thm}\label{t:functSN}
Let $(X_{i})$ be a linear process defined by
$$ X_{i} = \sum_{j=0}^{\infty}C_{j}Z_{i-j}, \qquad i \in \mathbb{Z},$$
where $(Z_{i})_{i \in \mathbb{Z}}$ is an i.i.d.~sequence of regularly varying random variables satisfying $(\ref{e:regvar})$ and $(\ref{eq:pq})$ with $\alpha \in (0,2)$. Assume conditions $(\ref{e:oceknula})$ and $(\ref{e:sim})$ hold. Let $(C_{i})_{i \geq 0}$ be a sequence of random variables, independent of $(Z_{i})$, satisfying conditions $(\ref{eq:InfiniteMAcond})$, $(\ref{e:momcond2})$, $(\ref{c:momcs})$ and $(\ref{e:modnew1})$.
       If $\alpha \in [1,2)$ suppose further condition $(\ref{eq:infmaTK3})$ holds.
Then
$$  \frac{1}{\zeta_{n}}  \sum_{i=1}^{\floor {n\,\cdot}}X_{i}=  \frac{V_{n}(\,\cdot\,)}{\sqrt{W_{n}(1)}} \dto \frac{C^{(1)}V(\,\cdot\,)}{\sqrt{C^{(2)}W(1)}} \qquad \textrm{as} \ n \to \infty,$$
in $D^{1}$ endowed with the Skorokhod $M_{2}$ topology, whit $\zeta_{n}=\sqrt{X_{1}^{2}+\ldots + X_{n}^{2}}$, and $L = (C^{(1)}V, C^{(2)}W)$ as given in Theorem~\ref{t:InfMA}.
\end{thm}
\begin{proof}
By Theorem~\ref{t:InfMA}, $L_{n} \dto L$ in $D^{2}$ with the weak $M_{2}$ topology. From this convergence relation, since
$$ \Pr[L \in D^{1} \times D_{m}([0,1], \mathbb{R})]=1,$$
by Corollary~\ref{c:M2cont} and the continuous mapping theorem
we obtain
\begin{equation}\label{e:mainconvSN}
      \widetilde{L}_{n} := (V_{n}, \widehat{W_{n}(1)}))
    \dto \widetilde{L} :=  (C^{(1)}V, C^{(2)}\widehat{W(1)})
\end{equation}
in $D^{1} \times D_{m}([0,1], \mathbb{R})$ with the weak $M_{2}$ topology, where for $a \in \mathbb{R}$, $\widehat{a}$ denotes the constant function defined by $\widehat{a}(t)=a$ for all $t \in [0,1]$.
Similar to Lemma~\ref{l:M1div} one can show that the function $g \colon D^{1} \times C^{\uparrow}_{0}([0,1], \mathbb{R}) \to D^{1}$ defined by
$$ g(x,y) = \frac{x}{\sqrt{y}},$$
is continuous
when $D^{1} \times C^{\uparrow}_{0}([0,1], \mathbb{R})$ is endowed with the weak $M_{2}$ topology and $D^{1}$ is endowed with the standard $M_{2}$ topology. Note that
$$ \Pr[\widetilde{L} \in D^{1} \times C^{\uparrow}_{0}([0,1], \mathbb{R})]=1,$$
 and hence from (\ref{e:mainconvSN}) by an application of the continuous mapping theorem we get $g(\widetilde{L}_{n}) \dto g(\widetilde{L})$ as $n \to \infty$, that is
$$ \frac{V_{n}(\,\cdot\,)}{\sqrt{W_{n}(1)}} \dto \frac{C^{(1)}V(\,\cdot\,)}{\sqrt{C^{(2)}W(1)}}$$
in $D^{1}$ with the $M_{2}$ topology.
\end{proof}




\section*{Acknowledgment}
 This work has been supported in part by University of Rijeka research grants uniri-prirod-18-9 and uniri-pr-prirod-19-16 and by Croatian Science Foundation under the project IP-2019-04-1239.


\end{document}